\documentclass{amsart}

\usepackage{amsmath,amsthm,amssymb}

\newtheorem{prop}{Proposition}[section]
\newtheorem{theorem}[prop]{Theorem}
\newtheorem{lemma}[prop]{Lemma}
\newtheorem{corr}[prop]{Corollary}
\theoremstyle{definition}
\newtheorem{definition}[prop]{Definition}
\theoremstyle{remark}
\newtheorem{remark}[prop]{Remark}
\newtheorem{example}[prop]{Example}

\DeclareMathOperator{\lt}{lt}
\DeclareMathOperator{\supp}{supp}
\DeclareMathOperator{\ran}{ran}
\DeclareMathOperator{\Ad}{Ad}

\DeclareMathOperator{\Ex}{Ex}
\DeclareMathOperator{\id}{id}

\newcommand{\mcal}[1]{\mathcal{#1}}
\newcommand{\mfrk}[1]{\mathfrak{#1}}

\newcommand{\inv}{^{-1}}

\newcommand{\cspn}{\overline{\text{span}}}
\newcommand{\sidehat}{^{\wedge}}
\newcommand{\doubledual}[1]{\,\widehat{\!\widehat{#1}}}
\newcommand{\unit}{^{(0)}}
\newcommand{\neghalf}{^{-\frac{1}{2}}}

\allowdisplaybreaks

\title{Locally Unitary Groupoid Crossed Products}
\author{Geoff Goehle}
\address{Mathematics and Computer Science Department, Stillwell 426,
  Western Carolina University, Cullowhee, NC 28723}
\email{grgoehle@email.wcu.edu}

\begin{document}

\begin{abstract}
We define the notion of a principal $S$-bundle where $S$
is a groupoid group bundle and show that there is a one-to-one
correspondence between principal $S$-bundles and elements of a sheaf
cohomology group associated to $S$.  We also define the notion of a
locally unitary action and show that the spectrum of the crossed
product is a principal $\widehat{S}$-bundle.  Furthermore, we prove
that the isomorphism class of the spectrum determines the exterior
equivalence class of the action and that every principal bundle can be
realized as the spectrum of some locally unitary crossed product.   
\end{abstract}

\maketitle

\section*{Introduction}

In this paper we study the connection between groupoid dynamical
systems and the spectrum of the crossed product.  
Our eventual goal will be to show there is a
strong link between the exterior equivalence class of a locally
unitary action and the isomorphism class of the spectrum
$(A\rtimes S)\sidehat$ as a principal bundle.  What's
more, there is a one-to-one correspondence between these principal
bundles and elements of a cohomology group, and this yields a
complete cohomological invariant for the exterior equivalence class of
locally unitary actions on algebras with Hausdorff spectrum.  
The primary objective of this work is to generalize the similar theory
of locally unitary group actions \cite{locunitary} to groupoids.  
However, the idea of using local triviality
conditions and cohomology classes to produce information about the
spectrum or primitive ideal space of crossed products has been implemented
in many contexts.  It is applied to, what turns out to be, a groupoid
setting in \cite{locunitarystab} and is often connected to actions on
algebras with continuous trace as in \cite{prelocunit} or \cite{cpct}.

The structure of the paper is as follows.  In Section
\ref{sec:principal} we introduce the notion of a principal bundle
associated to a groupoid group bundle.  This theory mirrors
the classic theory of principal group bundles.  In Section
\ref{sec:group-cross-prod} we introduce some basic results and
constructions concerning groupoid crossed products.  In Section
\ref{sec:unitary} we define what it means to be a unitary groupoid action 
and show that these actions are trivial in the sense that the
associated crossed product is given by a tensor product.  Finally, in
Section \ref{sec:locally-unitary} we describe locally unitary
actions.  Specifically, we show that if $\alpha$ is a locally unitary
action of the group bundle $S$ on a $C^*$-algebra with Hausdorff
spectrum $A$ then the spectrum of $A\rtimes_\alpha S$ is a 
principal $\widehat{S}$-bundle and the isomorphism class of the
spectrum determines the exterior equivalence class of the action.
Furthermore, we show that every principal bundle can be realized as
the spectrum of some locally unitary crossed product.

Before we begin in earnest it should be noted that the results of
this paper can be found, in more detail and with a great deal of
background material, in the author's thesis \cite{mythesis}.

\section{Principal Group Bundles}
\label{sec:principal}

In this section we will define the notion of a principal $S$-bundle
associated to a groupoid group bundle $S$.  The theory of principal
$S$-bundles turns out to be nearly identical to the classic theory of
principal group bundles.  
Consequently, we will only outline most of the proofs in
this section.   The following material is modeled off
\cite[Section 4.2]{tfb2}. 

\begin{remark}
A (groupoid) group bundle is a locally compact Hausdorff groupoid $S$ with
identical range and source maps.  Throughout this paper we will let
$S$ denote a second countable, locally compact Hausdorff groupoid
group bundle with abelian fibres 
and will denote both the range and the source map by $p$. 
\end{remark}

We begin with some definitions.  Recall that we may view any
continuous surjection $q:X\rightarrow Y$ as a ``topological'' bundle.
We will denote the fibres by $X_y := q\inv(y)$ for all $y\in Y$. 

\begin{definition}
Let $S$ be an abelian locally compact Hausdorff group bundle with
bundle map $p$.  Suppose
$X$ is a locally compact Hausdorff bundle over $S\unit$ with bundle
map $q$.  Furthermore, 
suppose there is an open cover $\mcal{U}=\{U_i\}_{i\in I}$ of $S\unit$ such
that for each $i\in I$ there is a homeomorphism $\phi_i:
q\inv(U_i)\rightarrow p\inv(U_i)$ with $p\circ \phi_i = q$.
Finally, suppose that for all $i,j\in I$ there is a section
$\gamma_{ij}$ of $S|_{U_{ij}} = p\inv(U_{ij})$ such that 
\[
\phi_i\circ \phi_j\inv(s) = \gamma_{ij}(p(s))s
\]
for all $s\in S|_{U_{ij}}$.  Such a bundle is called a {\em
  principal $S$-bundle with trivialization $(\mcal{U},\phi,\gamma)$.}  The
maps $\phi=\{\phi_i\}$ are referred to as {\em trivializing maps} and
the sections $\gamma= \{\gamma_{ij}\}$ are referred to as {\em
  transition maps.}
\end{definition}

\begin{definition}
\label{def:18}
Suppose $q:X\rightarrow S\unit$ and $r:Y\rightarrow S\unit$ are both
principal $S$-bundles with trivializations $(\mcal{U},\phi,\gamma)$
and $(\mcal{V},\psi,\eta)$ respectively.  Let $\mcal{W}$ be some
common refinement of $\mcal{U}$ and $\mcal{V}$ with refining maps $\rho$
and $\sigma$ respectively. Furthermore, suppose 
$\Omega:X\rightarrow Y$ is a homeomorphism such that $r\circ\Omega = q$ and
that for all $W_i\in \mcal{W}$
$\beta_i:W_i\rightarrow S$ is a section of $p$ 
such that for all $s\in p\inv(W_i)$
\[
\psi_{\sigma(i)} \circ \Omega \circ \phi_{\rho(i)}\inv(s) = \beta_i(p(s))s.
\]
Then $(\mcal{W},\Omega,\beta)$ is an $S$-bundle {\em isomorphism} of
$X$ onto $Y$.  
\end{definition}

\begin{remark}
Suppose we have a principal $S$-bundle $X$ with trivializations
$(\mcal{U},\phi,\gamma)$ and $(\mcal{V},\psi,\eta)$.  We say that the
trivializations are equivalent if the identity map forms an
isomorphism as in Definition \ref{def:18}.  We then say that $X$ is a
{\em principal $S$-bundle} if it is equipped with a maximal, pairwise
equivalent collection of trivializations called an atlas.  We say that
two principal $S$-bundles are isomorphic if they are isomorphic with
respect to any pair of trivializations in their respective atlases.  
This is the usual method of dealing with locally trivial bundles and
we won't make much of a fuss about it.  
\end{remark}

Next, we would like to mimic the group case and characterize the set
of all principal $S$-bundles by classes in some cohomology group.  We
will be using sheaf cohomology as defined and developed in
\cite[Section 4.1]{tfb}.

\begin{prop}
\label{prop:26}
Let $S$ be an abelian locally compact Hausdorff group bundle and for
$U$ open in $S\unit$ let $\mcal{S}(U) = \Gamma(U,S)$ be the set of
continuous sections from $U$ into $S$.  Then
$\mcal{S}$ is an abelian sheaf and as such gives rise to a sheaf cohomology 
$H^n(S\unit;\mcal{S})$ which we shall denote by $H^n(S)$.
\end{prop}

\begin{proof}
Straightforward arguments show that $\mcal{S}$ is a pre-sheaf with
$\rho_{U,V}:\mcal{S}(U)\rightarrow \mcal{S}(V)$ given by restriction.  
Now suppose we have an open set $U\subset S\unit$ and a decomposition
$U = \bigcup_{i\in I} U_i$ of $U$ into open sets $U_i$.   Furthermore, suppose
we have $\gamma_i\in \Gamma(U_i,S)$ for all $i\in I$ and for all
$i,j\in I$ 
\[
\rho_{U_i,U_{ij}}(\gamma_i) = \rho_{U_j,U_{ij}}(\gamma_j).
\]
Tracing through the definitions we see that each $\gamma_i$ is a
continuous section on $U_i$ such that the $\gamma_i$ agree on
overlaps.  Therefore, we can define a continuous section $\gamma$ on $U$
in a piecewise fashion so that $\rho_{U,U_i}(\gamma) = \gamma_i$.
Furthermore, it is clear that $\gamma$ is uniquely determined by the
$\gamma_i$.  Thus $\mcal{S}$ is a sheaf of groups on $S\unit$ which is
obviously abelian.   
\end{proof}

At this point we can build the desired correspondence between principal
$S$-bundles and elements of $H^1(S)$.  

\begin{theorem}
\label{prop:principcohom}
Suppose $S$ is an abelian locally compact Hausdorff group bundle.
There is a one-to-one correspondence between the isomorphism classes
of principal $S$-bundles and elements of the sheaf cohomology group
$H^1(S)$.  Given a principal bundle $X$ with trivialization 
$(\mcal{U},\phi,\gamma)$ the cohomology class in $H^1(S)$ associated
to $X$ is realized by the cocycle $\gamma$.  
\end{theorem}

\begin{proof}
Because this proof is so similar to the corresponding proof for
principal group bundles we will limit ourselves to sketching an
outline.  First, suppose $X$ is a principal $S$-bundle and pick a trivialization
$(\mcal{U},\phi,\gamma)$.  Then $\gamma = \{\gamma_{ij}\}$ turns out
to be a cocycle and as such we can use $\gamma$ to define a
class $[\gamma]\in H^1(S\unit;\mcal{S})$.  It is straightforward to
show that $[\gamma]$ is independent of which trivialization we choose
for $X$.  Now let $Y$ be another bundle isomorphic to $X$.  Suppose
$(\mcal{U},\phi,\gamma)$ is a trivialization for $X$,  
$(\mcal{V},\psi,\eta)$ a trivialization for $Y$, and let
$(\mcal{W},\Omega,\beta)$ be an isomorphism from $X$ to $Y$.  
By passing to a common
refinement we can assume, without loss of generality, that $\mcal{U} =
\mcal{V}=\mcal{W}$.  Then, for all $u\in U_{ij}$, simple calculations
show that $\eta_{ij}(u)\beta_j(u) = \beta_i(u)\gamma_{ij}(u)$.
Hence $\gamma\inv\eta$ is a boundary and
therefore $[\gamma]=[\eta]$ in $H^1(S\unit;\mcal{S})$.  
This shows that the map
$X\mapsto [\gamma]$ is a well defined function from the set of isomorphism
classes of principal $S$-bundles into $H^1(S)$.  

Next we are going to construct an inverse map.  Suppose $c\in
H^1(S\unit;\mcal{S})$ is realized by $\gamma\in Z^1(\mcal{U},\mcal{S})$
for some open cover $\mcal{U}$.  Let $C = \coprod_i p\inv(U_i)$ be the
disjoint union of the $p\inv(U_i)$ and denote elements of $C$ by
$(s,i)$ where $s\in p\inv(U_i)$.  Define a relation on $C$
by $(s,i)\equiv (t,j)$ if and only if $p(s)=p(t)=u$ and $s =
\gamma_{ij}(u)t$.  Elementary calculations using the cocycle identity 
show that $\equiv$ is an equivalence relation.  Let
$X_\gamma$ be the quotient of $C$ by $\equiv$ with
equivalence classes denoted by $[s,i]$ for $(s,i)\in C$
and associated quotient map $Q$.  Since the map $(s,i)\mapsto p(s)$ 
is constant on $[s,i]$, we can factor it through $Q$
to obtain a continuous surjection $q:X_\gamma\rightarrow S\unit$.
As it turns out, $\equiv$ is trivial on
$p\inv(U_i)\subset C$ and 
$\phi_i=Q|_{p\inv(U_i)}:p\inv(U_i) \rightarrow q\inv(U_i)$ is a
homeomorphism.  Because $X_\gamma$ is locally homeomorphic to $S$, it
follows that $X_\gamma$ is locally compact
Hausdorff and that we can view $X_\gamma$ as a bundle over $S\unit$ with
bundle map $q$. Furthermore, straightforward computations show
$X_{\gamma}$ is a principal $S$-bundle with trivialization 
$(\mcal{U},\phi,\gamma)$, and that the cohomology
class associated to $X_\gamma$ is $[\gamma] = c$.  

We must show that our map is well defined in the sense that if we
choose two different realizations of $c$ we end up with isomorphic
principal bundles.  Let $\eta = \{\eta_{ij}\}$ be some other cocycle
which implements $c$ on an open cover $\mcal{V}$.  Since
$[\eta]=[\gamma]=c$ we can pass to some common refinement of
$\mcal{U}$ and $\mcal{V}$, say $\mcal{W}$ with refining maps $r$ and
$\rho$ respectively, and find continuous
sections $\beta_i\in \Gamma(W_i,S)$ such that 
\[
\eta_{\rho(i)\rho(j)} \beta_j = \beta_i \gamma_{r(i)r(j)}.
\]
We define $\Omega : X_\gamma\rightarrow X_\eta$ locally by
$\Omega([s,r(i)]) = [\beta_i(p(s)) s, \rho(i)]$.  Some basic
arguments show that $\Omega$ is well defined and that
$(\mcal{W},\Omega,\beta)$ is an isomorphism from $X_\gamma$ onto
$X_\eta$.  Thus we have constructed a well defined map
$[\gamma]\mapsto X_\gamma$ from
$H^1(S)$ into the set of isomorphism classes of principal
$S$-bundles.  Furthermore, it is clear that this map is a right
inverse for $X\mapsto [\gamma]$.  Another simple argument shows that it is
also a left inverse so that we have the desired correspondence. 
\end{proof}

We continue our exploration of principal $S$-bundles by showing that
they are equivalent to a certain class of principal $S$-spaces.
First, observe that only transitive groupoids can have
transitive actions.  Thus, we make the following definition for
groupoid actions which are as transitive as they can be.   

\begin{definition}
\label{def:1}
Suppose $G$ is a locally compact Hausdorff groupoid acting on a
locally compact Hausdorff space $X$.  Then we say the action is {\em
  orbit transitive} if $G\cdot x = G\cdot y$ in $X/G$ whenever $G\cdot
r(x) = G\cdot r(y)$ in $G\unit/G$. 
\end{definition}

\begin{remark}
Recall that a $G$-space $X$ is {\em proper} if
the map $(\gamma,x)\mapsto (\gamma\cdot x, x)$ from $G\ltimes X$ into
$X\times X$ is proper.  The action is {\em principal} if it is
both free and proper. 
\end{remark}

Next we have the following construction.

\begin{prop}
\label{prop:27}
Suppose $X$ is a principal $S$-bundle with trivialization
$(\mcal{U},\phi,\gamma)$.  Define the range map on $X$ to be its bundle
map $q$.  Then, for $s\in S$, $x\in X$ such that $p(s)=q(x)\in U_i$, 
\begin{equation}
s\cdot x = \phi_i\inv(s\phi_i(x))
\end{equation}
defines a continuous action of $S$ on $X$.  Furthermore the following
hold:
\begin{enumerate}
\item The action of $S$ on $X$ is principal.
\item For all $i$ the map $\phi_i$ is equivariant with respect to this
  action and the action of $S$ on itself by left multiplication. 
\item The action of $S$ on $X$ is orbit transitive.  
\end{enumerate}
\end{prop}

\begin{proof}
Elementary calculations show that the action is well defined on
overlaps, respects the groupoid operations, and is continuous.  

Part {\bf (a)}: Suppose $s\cdot x = x$ for $s\in S$ and $x\in X$.  
Then for some $i$ we have
$\phi_i\inv(s\phi_i(x)) = x$ so that $s \phi_i(x) = \phi_i(x)$.  It
follows that $s\in S\unit$ and that the action is free.  Now suppose
$\{x_l\}$ and $\{s_l\}$ are nets in $X$ and $S$, respectively, so that
$x_l\rightarrow x$ and $s_l\cdot x_l\rightarrow y$.  We can pass to a
subnet and assume that $p(s)=q(x)=q(y) \in U_i$ and 
$p(s_l)=q(x_l)\in U_i$ for all $l$.  In this case $s_l\phi_i(x_l)
\rightarrow \phi_i(y)$ and, combining this with the fact that
$\phi_i(x_l)\rightarrow \phi_i(x)$, we have $s_l\rightarrow
\phi_i(y)\phi_i(x)\inv$.  It follows quickly that the action of $S$ on $X$ is
proper, and therefore principal.  

Part {\bf (b)}:  Suppose $s\in S$ and $x\in X$ such that
$p(s)=q(x)=u\in U_i$.  Then 
\[
\phi_i(s\cdot x) = \phi_i(\phi_i\inv(s\phi_i(x))) = s\phi_i(x).
\]

Part {\bf (c)}:  Suppose $x,y\in X$ such that $q(x) = q(y)\in U_i$  
and let $s = \phi_i(y)\phi_i(x)\inv$.  Then we are done since
$s\cdot x = \phi_i\inv(\phi_i(y)\phi_i(x)\inv\phi_i(x)) = y$.
\end{proof}

This next proposition shows that we can view principal $S$-bundles as
particularly nice $S$-spaces.  It is often useful think of principal
$S$-bundles in this manner.  

\begin{theorem}
\label{thm:principalspace}
Suppose $S$ is an abelian locally compact Hausdorff group bundle and
$X$ is a locally compact Hausdorff space.
Then $X$ is a principal $S$-bundle if and only if $X$ is a principal,
orbit transitive, $S$-space such that the range map on $X$
has local sections. 
\end{theorem}

\begin{proof}
If $X$ is a principal $S$-bundle then let $(\mcal{U},\phi,\gamma)$ be
a trivialization of $X$ and let $S$ act on $X$ as in Proposition
\ref{prop:27}.  On $U_i$ define $\sigma_i: U_i\rightarrow X$
by $\sigma_i(u) = \phi_i\inv(u)$.  It is easy to see that $\sigma_i$
is a continuous section of $q$ on $U_i$. 

Now, suppose $S$ acts on $X$ as in the statement of the theorem and
that $\mcal{U}$ is an open cover of $S\unit$ such that there
are local sections $\sigma_i:U_i\rightarrow X$ of $q$.  We define
$\psi_i:p\inv(U_i)\rightarrow q\inv(U_i)$ by $\psi_i(s) = s\cdot
\sigma_i(p(s))$.  It is clear
that $\psi_i$ is continuous.  A straightforward argument shows that
$\psi_i$ is injective because the action is free, and that $\psi_i$ is
surjective because the action is orbit transitive.  Now suppose
$\psi_i(s_l)\rightarrow \psi_i(s)$.  By
definition we have $s_l \cdot \sigma_i(p(s_l)) \rightarrow s\cdot
\sigma_i(p(s))$.  Furthermore $q(s_l\cdot \sigma_i(p(s_l))) = p(s_l)$
for all $l$ and $q(s\cdot \sigma_i(p(s)))=p(s)$.  Since $q$ is continuous,
we have $p(s_l)\rightarrow p(s)$ and therefore
$\sigma_i(p(s_l))\rightarrow \sigma_i(p(s))$.  Since the action of $S$
is proper, this implies that we can pass to a subnet, relabel, and find
$t$ such that $s_l\rightarrow t$.  However, using the continuity of
the action, this implies $s_l\cdot\sigma_i(p(s_l))\rightarrow t\cdot
\sigma_i(p(s))$.  Using the fact that $X$ is Hausdorff and the action
is free, we have $s=t$.  It follows that $\psi_i$ is a
homeomorphism and we define the trivializing maps to be $\phi_i =
\psi_i\inv$.  Next, we need to compute the transition functions.
Suppose $s\in p\inv(U_{ij})$.  Then  
\begin{align*}
\phi_i\circ \phi_j\inv(s) &= \psi_i\inv\circ \psi_j(s) 
= \psi_i\inv(s\cdot\sigma_j(p(s))) \\
&= \psi_i\inv(\gamma_{ij}(p(s))s\cdot \sigma_i(p(s))) =
\gamma_{ij}(p(s))s
\end{align*}
where $\gamma_{ij}(u)$ is the unique element of $S$ such that 
$\gamma_{ij}(u) \cdot \sigma_i(u) = \sigma_j(u)$. We know $\gamma_{ij}(u)$ is 
guaranteed to exist because the action is orbit transitive and
that $\gamma_{ij}(u)$  is unique because the action is free.  It is
simple enough to show that $\gamma_{ij}$ is continuous.  Hence $X$ is
a principal $S$-bundle with trivialization $(\mcal{U},\phi,\gamma)$.  
\end{proof}

This next proposition is nice because it frees our idea of principal
bundle isomorphism from the hassle of having to keep track of local
trivializations.  It is also mildly remarkable that $\Omega$ is not
required to be a homeomorphism, or even a bijection.  

\begin{prop}
Suppose $X$ and $Y$ are principal $S$-bundles. Then $X$ and $Y$ are
isomorphic if and only if there exists a continuous map $\Omega: X\rightarrow
Y$ which is $S$-equivariant with respect to the actions of
$S$ on $X$ and $Y$. 
\end{prop}

\begin{proof}
Let $X$ and $Y$ be as above with bundle maps $q$ and $r$, and
trivializations $(\mcal{U},\phi,\gamma)$ and $(\mcal{V},\psi,\eta)$,
respectively.  Elementary calculations show that if 
$(\mcal{W},\Omega,\beta)$ is a  principal bundle isomorphism from $X$
to $Y$ then $\Omega$ is equivariant.  

Now suppose $\Omega:X\rightarrow Y$ is a continuous equivariant map.
By passing to a common refinement we may assume without
loss of generality that $\mcal{U}=\mcal{V}$.  Given $U_i$ let $\Omega_i
= \psi_i \circ \Omega \circ \phi_i\inv$.  Since each of its component
maps preserves fibres, $\Omega_i$ does as well, and therefore
$\Omega_i|_{S_u}$ maps $S_u$ into $S_u$ for $u\in U_i$.  If $s,t\in S_u$ then 
\[
\Omega_i(st )=
\psi_i\circ\Omega(s\cdot\phi_i\inv(t)) = \psi_i(s\cdot \Omega(\phi_i\inv(t)) 
= s\Omega_i(t).
\]
Observe that given an abelian group homomorphism $h:H\rightarrow H$ such that 
$h(st)=sh(t)$ for all $s,t\in H$ we have  $h(s) = h(es) = h(e)s$.
Hence, $h$ is actually just left multiplication by $h(e)$.  Applying
this to the current situation we find that $\Omega_i|_{S_u}$ is 
left multiplication by $\Omega_i(u)$ on $S_u$.  Define $\beta_i$ on
$U_i$ by $\beta_i(u) = \Omega_i(u)$.  The function $\beta_i$ is a section
of $S$ on $U_i$ which is continuous because $\Omega_i$ is continuous.  Since
$\Omega_i$ is defined by left multiplication against
$\beta_i$, it follows immediately that $\Omega_i$ has a continuous
inverse given by left multiplication against $\beta_i\inv$.  Thus
$\Omega_i$ is a homeomorphism.  It is
straightforward to show that this implies that $\Omega$ is a
homeomorphism.  Furthermore, we know that for $s\in q\inv(U_i)$ 
\[
\psi_i \circ\Omega\circ\phi_i\inv(s) = \Omega_i(s) = \beta_i(p(s)) s.
\]
Hence $(\mcal{U},\Omega,\beta)$ is a principal bundle isomorphism of $X$ onto $Y$.
\end{proof}

\begin{remark}
\label{rem:1}
It is philosophically important to see that the theory of principal
$S$-bundles is an extension of the classical theory of principal group bundles.
Suppose $H$ is an abelian locally compact Hausdorff group and $X$ and $Y$ are
locally compact Hausdorff spaces.  Let $S = Y\times H$ be the trivial group
bundle. Then it is not difficult to show that 
$X$ is a principal $H$-bundle over $Y$ if and only if $X$ is a principal
$S$-bundle.  What's more, $H^n(S)\cong H^n(Y;H)$ and under this
identification $X$ generates the same cohomology class when viewed as
either a principal $S$-bundle or a principal $H$-bundle.
\end{remark}

\subsection{Locally $\sigma$-trivial Spaces}

As observed in Remark \ref{rem:1}, 
principal ``group bundle bundle'' theory is a natural extension of
classical principal group bundle theory.  
The real question is if there are principal
$S$-bundles which are not generated by principal $H$-bundles.
Fortunately, principal $S$-bundles are also an extension of the
notion of $\sigma$-trivial spaces as defined in \cite{locunitarystab}
and there are nontrivial examples given there.  First, however, we
define what it means to be $\sigma$-trivial.   

\begin{definition}
\label{def:20}
Suppose the abelian locally compact Hausdorff group $H$ acts on the
locally compact Hausdorff space $X$ and the stabilizers vary
continuously in $H$ with respect to the Fell topology.  We
shall say that $X$ is a locally $\sigma$-trivial space if $X/H$ is
Hausdorff and if every $x\in X$ has a $H$-invariant neighborhood $U$
such that there exists a homeomorphism 
$\phi:U\rightarrow (U/H\times H)/\cong$ where 
\[
(H\cdot x,s)\cong(H\cdot y,t)\ \text{if and only if $H\cdot x=H\cdot y$
  and $st\inv\in H_x.$}
\]
Furthermore we require that  
\begin{enumerate}
\item If $x\in U$ then $\phi(x) = [H\cdot x, s]$ for some $s\in H$
  and, 
\item If $x\in U$, $s\in H$ and $\phi(x) = [H\cdot x, t]$ then 
$\phi(s\cdot x) = [H\cdot x, st]$.
\end{enumerate}
\end{definition}

Our goal will be to construct an abelian group
bundle $S$ associated to $G$ and $X$.  Because this construction is
unimportant for what follows, we will omit some of the details. 

\begin{prop}
\label{prop:30}
Suppose the abelian locally compact Hausdorff group $H$ acts on the
locally compact Hausdorff space $X$.  Furthermore, suppose the
stabilizers vary continuously in $H$ and that $X/H$ is Hausdorff.
Define $S_{(X,H)} := (X/H\times H)/\cong$, often denoted $S$, 
where $\cong$ is as in Definition \ref{def:20}.  Then $S$ is
an abelian locally compact Hausdorff group
bundle with a Haar system whose unit space can be identified with $X/H$.  
The bundle map is given by $p([H\cdot x, s]) = H\cdot x$ and the
operations are 
\[
[H\cdot x, s][H\cdot x, t] := [H\cdot x, st],\quad\text{and}\quad
[H\cdot x, s]\inv := [H\cdot x, s\inv].
\]
The fibre $S_{H\cdot x}$ over $H\cdot x$ is (isomorphic to) $H/H_x$.  
\end{prop}

\begin{proof}
Define $S$ as in the statement of the proposition. A relatively simple
argument shows that $S$ is locally compact Hausdorff and that the
quotient map $Q:X/H\times H\rightarrow S$ is open.  It is similarly
straightforward to show that,
given the operations above, $S$ is algebraically a group bundle over $X/H$ with
bundle map $p$ given by $p([H\cdot x, s])= H\cdot x$.  We now assert
that the operations are continuous and that $p$ is open.  This can be
proved by using the fact that $Q$ is open.  It then
follows from \cite[Lemma 1.3]{renaultgcp} that $S$ has a Haar system.  
Given $H\cdot x\in X/H$ we have $S_{H\cdot x} = \{[H\cdot x,s ]: s\in
H\}$.  We can define a continuous surjective homomorphism 
$\phi:H\rightarrow S_{H\cdot x}$ by
$\phi(s) = [H\cdot x, s]$.  Straightforward arguments show
that $\phi$ is open, and hence it is clear from the definition of
$\cong$ that $\phi$ factors to an isomorphism of
$H/H_x$ with $S_{H\cdot x}$.  Since $H/H_x$ is clearly abelian, this
proves that $S$ has abelian fibres and we are done.  
\end{proof}

The reason we went through all of this rigmarole is that given a
locally $\sigma$-trivial system $(H,X)$ we would like to show that $X$ is a
principal $S_{(H,X)}$-bundle.  

\begin{prop}
\label{prop:31}
Suppose $H$ is an abelian locally compact Hausdorff group acting on a
locally compact Hausdorff space $X$ with continuously varying
stabilizers such that $X/H$ is Hausdorff.  If $X$ is locally 
$\sigma$-trivial then $X$ is a principal $S_{(H,X)}$-bundle. 
\end{prop}

\begin{proof}
Let $q:X\rightarrow X/H$ be the quotient map.  We know from Definition
\ref{def:20} that if $x\in X$ then there is an $H$-invariant
neighborhood $U$ that is homeomorphic to $U/H\times H/\cong$.  If we
let $V = U/H$ then $V$ is an open neighborhood of $H\cdot x$ and
$q\inv (V) = U$. Let $p$ be the bundle map for $S$ and
observe that 
\[
p\inv(V) = \{[H\cdot x,s]\in S:H\cdot x\in V\} =  U/H\times H/\cong. 
\]
Thus we have a homeomorphism $\phi_{V}:q\inv(V)\rightarrow p\inv(V)$.  
Find one of these neighborhoods for every
$x\in X$ and use them to build an open cover $\mcal{V}$ of $X/H$.  
For $V_i$ in this open cover let $\phi_i= \phi_{V_i}$.  Given
$V_{ij}$ define $\gamma_{ij}:U_{ij}\rightarrow S$ by
$\gamma_{ij}(H\cdot x) = \phi_i\circ\phi_j\inv([H\cdot x, e])$ where
$e$ is the unit in $H$.  It is
clear that $\gamma_{ij}$ is a continuous section on $V_{ij}$.  Since
$\gamma_{ij}$ is a section, we can find a function 
$\tilde{\gamma}_{ij}$ from $V_{ij}$ into $H$ such that
$\gamma_{ij}(H\cdot x) = [H\cdot x, \tilde{\gamma}_{ij}(H\cdot x)]$.  
Suppose $[H\cdot x, s]\in p\inv(V_{ij})$.  Then, using the equivariance
condition of Definition \ref{def:20}, we have 
\begin{align*}
\phi_i\circ\phi_j\inv([H\cdot x, s]) &= 
\phi_i(s\cdot \phi_j\inv([H\cdot x,e])) = [H\cdot x,
\tilde{\gamma}_{ij}(H\cdot x)s] \\ 
&= \gamma_{ij}(H\cdot x) [H\cdot x,s].
\end{align*}
Thus $X$ is a principal $S$-bundle with trivialization
$(\mcal{V},\phi,\gamma)$.  
\end{proof}

\begin{remark}
The upshot of Proposition \ref{prop:31} is that every nontrivial example of a
locally $\sigma$-trivial space given in \cite{locunitarystab} is an
example of a nontrivial principal $S$-bundle.  This shows that such
bundles do exist, and it then follows from 
Theorem \ref{prop:principcohom} that there are group bundles with
nontrivial cohomology.  
\end{remark}

\begin{remark}
In \cite{locunitarystab} a $\sigma$-trivial space is said to be 
locally liftable if given a 
continuous section $c:U\rightarrow S$ then there exists $V\subset U$
and  a
continuous map $\tilde{c}:V \rightarrow H$ such that $c(H\cdot x) = [H\cdot x,
\tilde{c}(H\cdot x)]$.  Locally $\sigma$-trivial {\em bundles} are defined to
be locally $\sigma$-trivial spaces which are also locally liftable.
The reason for this extra requirement has to do with finding a
cohomological invariant for $X$. Let $\mcal{T}$ be the sheaf
defined for $U\subset X/H$ by $\mcal{T}(U) = C(U,H)$ and 
$\mcal{R}$ be the subsheaf of $\mcal{T}$ where $\mcal{R}(U)$ is the
subset of $\mcal{T}(U)$ such that $f(H\cdot x)\in H_x$ for all $x$.  Then sheaf
cohomological considerations will show that we can construct a
quotient sheaf $\mcal{T}/\mcal{R}$ and an associated cohomology
$H^n(X/H; \mcal{T}/\mcal{R})$.  Given a $\sigma$-trivial space one
would like to use the transition maps $\gamma_{ij}$, as defined in
the proof of Proposition \ref{prop:31}, to construct an element of
$H^1(X/H;\mcal{T}/\mcal{R})$.  The problem is that while $\gamma_{ij}$
is a continuous section of $U_{ij}$ into $S$ the associated map
$\tilde{\gamma}_{ij}:U_{ij}\rightarrow H$ may not be continuous.  If
$\tilde{\gamma}_{ij}$ is not continuous then it doesn't define an
element of $\mcal{T}(U_{ij})$ and we cannot construct the appropriate
cohomology element.  However, if $X$ is required to be locally
liftable then, by passing to a smaller open set, we can guarantee that
$\tilde{\gamma}_{ij}$ is a continuous function.  As such it defines an
element of $\mcal{T}(U_{ij})$ and hence a cohomology element in
$H^1(X/H;\mcal{T}/\mcal{R})$.  In fact, it is shown in
\cite{locunitarystab} that this construction leads to a one-to-one
correspondence between locally $\sigma$-trivial {\em bundles} with a
fixed orbit space $X/H$ and $H^1(X/H;\mcal{T}/\mcal{R})$.  

This is an artificial restriction in our setting.  The
$\gamma_{ij}$ can always be used to define an element of
$H^1(S_{(X,H)})$, regardless of whether $\sigma$ is locally liftable or
not.  It is comforting to observe the following, however.  Let
$\mcal{S}$ be the sheaf of local sections of $S$ so that $H^n(S) =
H^n(X/H;\mcal{S})$ by definition.  It is straightforward to show that if
$\sigma$ is locally liftable then we get a short exact sequence of
sheaves 
\[
0\rightarrow \mcal{R}\rightarrow \mcal{T}\rightarrow
\mcal{S}\rightarrow 0
\]
and that $H^n(X/H;\mcal{S})$ is naturally isomorphic to
$H^n(X/H,\mcal{T}/\mcal{R})$.  Furthermore, once one sorts out all of
the various constructions, it is clear that the different
cohomological invariants of a $\sigma$-trivial bundle are identified
under this isomorphism. 
\end{remark}

\section{Groupoid Crossed Products}
\label{sec:group-cross-prod}

For the rest of the paper we will assume that $G$, or $S$, is a second
countable locally compact Hausdorff groupoid, or group bundle, with a
Haar system.  We will let $A$ be a separable $C_0(G\unit)$-algebra and
$\mcal{A}$ be its associated upper-semicontinuous bundle (usc-bundle).
In general we will use the $C_0(X)$-algebra 
notation from \cite[Appendix C]{tfb2} or \cite[Section 3.1]{mythesis}
and refer readers to these sources as references.  Next, let 
$(A,G,\alpha)$ be a groupoid dynamical system as defined in
\cite{renaultequiv} or \cite[Section 3.2]{mythesis}.  We will assume that
the reader is familiar with groupoid dynamical systems, although we
take the time here to establish some of the basics.  First, recall
that $\alpha$ is given by a collection of isomorphisms
$\{\alpha_\gamma\}_{\gamma\in G}$ such that
$\alpha_\gamma:A(s(\gamma))\rightarrow A(r(\gamma))$, 
$\alpha_{\gamma\eta}=\alpha_\gamma\circ\alpha_\eta$ whenever
$\gamma$ and $\eta$ are composable, and $\gamma\cdot
a := \alpha_\gamma(a)$ defines a continuous action of $G$ on
$\mcal{A}$.  We then form the convolution algebra
$\Gamma_c(G,r^*\mcal{A})$ given the operations
\[
f*g(\gamma) = \int_G f(\eta)\alpha_\eta(g(\eta\inv\gamma))
d\lambda^{r(\gamma)}(\eta) \quad\text{and}\quad
f^*(\gamma) = \alpha_\gamma(g(\gamma\inv)^*).
\]

Next, suppose we have a representation $(U,X*\mfrk{H},\mu)$ of $G$ and
a $C_0(G\unit)$-linear representation $\pi$ of $A$ on
$L^2(X*\mfrk{H},\mu)$.  It follows from some
fairly heavy representation theory \cite[Section 7]{renaultequiv},
\cite[Section 3.3]{mythesis} that there exists a collection of
representations $\pi_u:A(u)\rightarrow B(\mcal{H}_u)$ such that
$\pi(a)$ is equal to the direct integral
$\int_{G\unit}^\oplus \pi_u(a(u))\,d\mu(u)$ for all $a\in A$.  This
decomposition is an essential aspect of $C_0(G\unit)$-linear
representations of $A$ and will be used in Section \ref{sec:unitary}.
Next, recall that we can use the quasi-invariant measure $\mu$ to
induce a measure $\nu = \int_{G\unit}\lambda^u d\mu(u)$ on $G$.  We
say that $(\pi,U,X*\mfrk{H},\mu)$ is a {\em covariant} representation
if the relation 
\begin{equation}
\label{eq:cov}
U_\gamma \pi_{s(\gamma)}(a) =
\pi_{r(\gamma)}(\alpha_\gamma(a))U_\gamma\quad\text{for all $a\in
  A(s(\gamma))$}
\end{equation}
holds $\nu$-almost everywhere on $G$.  We can then form the {\em
  integrated} representation $\pi\rtimes U$ of
$\Gamma_c(G,r^*\mcal{A})$ on $L^2(X*\mfrk{H},\mu)$ by 
\begin{equation}
\label{eq:36}
\pi\rtimes U(f)h(u) = \int_G \pi_u(f(\gamma))U_\gamma
h(s(\gamma))\Delta(\gamma)\neghalf d\lambda^u(\gamma)
\end{equation}
where $\Delta$ is the modular function associated to $\mu$. 
It is then either a definition or a theorem that the crossed product
$A\rtimes G$ is the completion of $\Gamma_c(G,r^*\mcal{A})$ with
respect to the universal norm arising from the integrated covariant
representations.  

\begin{remark}
\label{rem:4}
This construction mirrors the construction of groupoid
$C^*$-algebras.  In the groupoid case the integrated form of a unitary
representation is given on $C_c(G)$ by 
\[
U(f)h(u) = \int_G f(\gamma)U_\gamma h(s(\gamma))
\Delta(\gamma)\neghalf d\lambda^u(\gamma)
\]
and $C^*(G)$ is the completion of $C_c(G)$ with respect to the
universal norm determined by these representations. 
\end{remark}

With the exception of this section we will generally be interested
only in actions of group bundles on $C^*$-algebras.  As such, most of
our crossed products will have extra structure.  Recall from
\cite[Proposition 1.2]{specpaper} that if $S$ is a group bundle and
$(A,S,\alpha)$ is a dynamical system then $A\rtimes S$ is a
$C_0(S\unit)$-algebra and that restriction from
$\Gamma_c(S,p^*\mcal{A})$ to $C_c(S_u,A(u))$ factors to an isomorphism
of the fibre $A\rtimes S(u)$ with the group crossed product
$A(u)\rtimes S_u$.  In fact, this identification is even more robust.
Recall that if $A$ is a $C_0(X)$-algebra and $U\subset X$ is open then
we define $A(U)$ to be the section algebra $\Gamma_0(U,\mcal{A})$
\cite[Section 1]{inducpaper}.  

\begin{prop}
\label{prop:88}
Suppose $(A,S,\alpha)$ is a groupoid dynamical system, $S$ is a group
bundle, and that $U$ is
an open subset of $S\unit$.  Then $A\rtimes_\alpha S(U)$ and
$A(U)\rtimes_\alpha S|_U$ are isomorphic as $C_0(U)$-algebras.  
\end{prop}

\begin{proof}
Let $C=S\unit\setminus U$ and 
recall that $A\rtimes S$ is a $C_0(S\unit)$-algebra.  It follows from
some general $C_0(X)$-algebra theory that $A\rtimes S(U)$ is isomorphic to the
ideal 
\[
I_C = \cspn\{\phi\cdot f : \phi\in
C_0(S\unit),f\in\Gamma_c(S,p^*\mcal{A}), \phi(C) = 0\}
\]
via the inclusion map $\iota_1:A\rtimes S(U)\rightarrow A\rtimes S$
where we view both spaces as section algebras of their associated bundle.  
Because $S$ acts trivially on its unit space, $U$ is $S$-invariant and
it follows from \cite[Theorem 3.3]{inducpaper} that the inclusion map
$\iota_2:\Gamma_c(S|_U,p^*\mcal{A})\rightarrow \Gamma_c(S,p^*\mcal{A})$
extends to an isomorphism of $A(U)\rtimes S|_U$ with $\Ex(U)$ where
$\Ex(U)$ is the closure of $\Gamma_c(S|_U,p^*\mcal{A})$ in $A\rtimes S$.
A standard approximation argument shows that $I_C = \Ex(U)$.  
Thus $\iota_2\inv\circ\iota_1$ yields the desired isomorphism.  The
fact that it is $C_0(U)$-linear follows from a computation.  
\end{proof}

\begin{remark}
\label{rem:2}
The identification made in Proposition \ref{prop:88} extends to the
spectrum.  Recall from \cite[Proposition C.5]{tfb2} that $(A\rtimes
S)\sidehat$ can be identified as a set with the disjoint union
$\coprod_{u\in S\unit} (A(u)\rtimes S_u)\sidehat$ and that there is a
continuous map $q:(A\rtimes S)\sidehat\rightarrow S\unit$ such that
$q(\pi) = u$ if and only if $\pi$ factors to a representation of
$A(u)\rtimes S_u$.  Suppose $U\subset S\unit$ is open.  Then each of
$q\inv (U)$, $(A\rtimes S(U))\sidehat$ and $(A(U)\rtimes
S|_U)\sidehat$ can be identified setwise with $\coprod_{u\in U}(A(u)\rtimes
S_u)\sidehat$.  It is a matter of sorting out definitions and applying
Proposition \ref{prop:88} to see that all three algebras induce the
same topology on the disjoint union.  
\end{remark}

Groupoid dynamical systems 
are fairly common and in particular arise naturally
from groupoid actions on spaces, which we shall demonstrate after
proving the following

\begin{lemma}
\label{lem:26}
Suppose $X$, $Y$  and $Z$ are locally compact Hausdorff spaces and that
$\sigma:Y\rightarrow X$ and $\tau:Z\rightarrow X$ are continuous
surjections. Let $Z*Y = \{(z,y)\in Z\times Y : \tau(z) =
\sigma(y)\}$.  Then the map $\iota:C_0(Z*Y)\rightarrow \tau^*(C_0(Y))$
such that $\iota(f)(z)(y) = f(z,y)$ is an isomorphism. 
Furthermore, $\iota(f)$ is compactly supported if $f$ is and $\iota$
preserves convergence with respect to the inductive limit topology.  
\end{lemma}

\begin{remark}
\label{rem:3}
Recall from \cite[Example C.4]{tfb2} that if $\sigma:Y\rightarrow X$
is a continuous surjection then $C_0(Y)$ has a $C_0(X)$-algebra
structure.  Furthermore, the fibre $C_0(Y)(x)$ is isomorphic to
$C_0(\sigma\inv(x))$ via restriction.
\end{remark}

\begin{proof}
Let $\mcal{C}$ be the usc-bundle associated to $C_0(Y)$ as a $C_0(X)$-algebra.  
Define $\iota: C_c(Z*Y)\rightarrow \Gamma_c(Z,\tau^*\mcal{C})$
by $\iota(f)(z)(y) = f(z,y)$.  It is straightforward to show that
$\iota(f)$ is a compactly supported section of $\tau^*\mcal{C}$.  
We need to see that
$\iota(f)$ is continuous.  We start by demonstrating this in a simpler
case. Suppose $g\in C_c(Z)$, $h\in C_c(Y)$ and define $g\otimes
h(z,y) = g(z)h(y)$ for all $(z,y)\in Z*Y$. Suppose $z_i\rightarrow z$.
Since $h\in C_c(Y)$, we can view $h$ as a continuous
section of $\mcal{C}$ with $h(x) = h|_{\sigma\inv(x)}$ for all $x\in
X$ and therefore $h(\tau(z_i))\rightarrow h(\tau(z))$
in $\mcal{C}$, since $\tau$ is continuous.  It follows quickly from the
fact that scalar multiplication is continuous 
that $\iota(f)$ is a continuous section.  

Now suppose we have $f\in C_c(Z*Y)$.  Since $Z*Y$ is closed in $Z\times Y$,
we can extend $f$ to the product space and then 
find $g_i^j\in C_c(Z)$ and $h_i^j\in C_c(Y)$ such that 
$k_i = \sum_j g_i^j\otimes h_i^j
\rightarrow f$ uniformly.  Let $z_i\rightarrow z$ and
observe that $\tau(z_i)\rightarrow \tau(z)$.  We will show
$\iota(f)(z_i)\rightarrow \iota(f)(z)$ using \cite[Proposition
C.20]{tfb2}.  Let $\epsilon > 0$ and choose $I$ such that
$\|k_I -f \|_\infty < \epsilon$.  Since sums of continuous functions
are continuous,
$\iota(k_I)(z_i)\rightarrow \iota(k_I)(z)$ by the previous paragraph.  
Furthermore, given $w\in Z$ we have 
\[
\|\iota(k_I)(w) -\iota(f)(w)\|_\infty  = 
\sup_{y\in \sigma\inv(\tau(w))} |k_I(w,y) - f(w,y)| 
\leq \|k_I - f\|_\infty < \epsilon.
\]
Since this is true for all $z_i$ and $z$, we are done. 

An elementary calculation now shows that $\iota$ is isometric and it
follows from an application of \cite[Proposition C.24]{tfb2} that
$\ran\iota$ is dense in $\Gamma_0(Z,\tau^*\mcal{C})$.  Since
$\iota:C_c(Z*Y)\rightarrow \tau^*C_0(Y)$ is an isometry mapping onto a
dense set we can extend $\iota$ to an isomorphism
$\iota:C_0(Z*Y) \rightarrow \tau^*C_0(Y)$.  It is now straightforward
to see that $\iota$ has the desired form on all of $C_0(Z*Y)$ and that
$\iota$ preserves convergence with respect to the inductive limit
topology. 
\end{proof}

Using Lemma \ref{lem:26} we can show that left translation yields a
groupoid dynamical system from any $G$-space.  

\begin{prop}
\label{prop:68}
Suppose a locally compact Hausdorff groupoid $G$ acts on a
locally compact Hausdorff space $X$.  Then $C_0(X)$ is a
$C_0(G\unit)$-algebra and there is an action of $G$ on $C_0(X)$ given
by $\lt_\gamma:C_0(r_X\inv(s(\gamma))) \rightarrow
C_0(r_X\inv(r(\gamma)))$ where 
\begin{equation}
\lt_\gamma(f)(y) = f(\gamma\inv\cdot y)
\end{equation}
for all $f\in C_0(r_X\inv(s(\gamma)))$ and $y\in r_X\inv(r(\gamma))$.
\end{prop}

\begin{proof}
The $C_0(G\unit)$-algebra structure on $C_0(X)$ arises from the range
map on $X$ as in Remark \ref{rem:3}. Let $\mcal{C}$ be the usc-bundle
associated to $C_0(X)$.  Define
$\lt_\gamma:C_0(r_X(s(\gamma)))\rightarrow C_0(r_X(r(\gamma)))$ as in
the statement of the proposition.  It is straightforward to show that
each $\lt_\gamma$ is an isomorphism and that the groupoid operations
are preserved.  The only difficult part is proving that the action is
continuous.  Suppose $f_i\rightarrow f$ in $\mcal{C}$ and
$\gamma_i\rightarrow \gamma$ in $G$ such that $u_i = p(f_i) = s(\gamma_i)$
for all $i$ and $u=p(f) = s(\gamma)$.  
Observe from Lemma \ref{lem:26} that there is an
isomorphism $\iota:C_0(G*X) \rightarrow r^*C_0(X)$ and find $g\in
C_0(G*X)$ such that $\iota(g)(u) = f$.  Define $\tilde{g}\in C_0(G*X)$
by $\tilde{g}(\gamma,x) = g(\gamma,\gamma\inv\cdot x)$.  A simple
calculation shows that $\iota(\tilde{g})(\eta) =
\lt_\eta(\iota(g)(\eta))$ for all $\eta\in G$.  In particular, since
$\iota(\tilde{g})$ is a continuous section, we must have
$\lt_{\gamma_i}(\iota(g)(\gamma_i)) \rightarrow
\lt_\gamma(\iota(g)(\gamma))$.  Finally, observe that 
\[
\|\lt_{\gamma_i}(\iota(g)(\gamma_i)) - \lt_{\gamma_i}(f_i)\| =
\|\iota(g)(\gamma_i)-f_i\| \rightarrow 0
\]
so that by \cite[Proposition C.20]{tfb2} we have
$\lt_{\gamma_i}(f_i)\rightarrow \lt_\gamma(f)$. 
\end{proof}

\begin{remark}
It is not particularly difficult to use Lemma \ref{lem:26} to 
show that if $X$ is a $G$-space
and $G\ltimes X$ is the transformation groupoid then
$C_0(X)\rtimes_{\lt} G$ is naturally isomorphic to the transformation
groupoid $C^*$-algebra $C^*(G\ltimes X)$.  However, since this fact has
little bearing on our current discussion, the proof has been omitted.  
\end{remark}

\begin{remark}
There is a converse to Proposition \ref{prop:68} for abelian
algebras.  Given a groupoid dynamical system $(C_0(X),G,\alpha)$ 
it follows from \cite[Proposition 1.1]{specpaper} that there is an
action of $G$ on $C_0(X)\sidehat = X$. It is straightforward to show,
once one sorts out all the definitions, that $\alpha$ is given by left
translation with respect to this action.  
\end{remark}

\begin{example}
If $X$ is a principal $S$-space then we showed in Proposition
\ref{prop:27} that there is a continuous action of $S$ on $X$.  Hence
there is an associated crossed product $C_0(X)\rtimes_{\lt} S$ which
will be essential in the latter half of Section \ref{sec:locally-unitary}.  
\end{example}

\section{Unitary Actions}

\label{sec:unitary}
In this section we will discuss what it means for a groupoid to act
trivially.  The main goal will be to show that if the action is
trivial then the crossed product reduces to a balanced tensor product.  As with
group crossed products, trivial actions are going to be defined by
unitaries.  

\begin{definition}
\label{def:51}
Suppose $S$ is a locally compact Hausdorff 
groupoid group bundle and $A$ is a
$C_0(S\unit)$-algebra.  Then a {\em unitary action} of $S$ on $A$ is
defined to be a collection $\{u_s\}_{s\in S}$ such that 
\begin{enumerate}
\item $u_s \in UM(A(p(s)))$ for all $s\in S$,
\item $u_{st} = u_su_t$ whenever $p(s)=p(t)$, and 
\item $s\cdot a := u_s a$ defines a continuous action of $S$ on $\mcal{A}$.
\end{enumerate}
The triple $(A,S,u)$ is called a unitary dynamical system.  
\end{definition}

\begin{remark}
If $u$ is a unitary action
of $S$ on $A$ then the restriction of $u$ to $S_v$ for $v\in S\unit$
gives a unitary action of $S_v$ on $A(v)$ in the sense of
\cite[Definition 2.70]{tfb2}.  Thus, Definition \ref{def:51} is really
just a ``bundled'' version of the notion of a unitary action of
a group on a $C^*$-algebra.  
\end{remark}

As with groupoid dynamical systems there is an ``unbundled''
definition.

\begin{prop}
\label{prop:76}
Suppose $(A,S,u)$ is a unitary dynamical system.  Then there is an
element $u\in UM(p^*A)$ such that $u(s) = u_s$ for all $s\in S$.
Conversely, if we have $u\in UM(p^*A)$ then there are elements $u_s\in
UM(A(p(s)))$ for all $s\in S$ and if $u_{st} = u_s u_t$ whenever
$p(s)=p(t)$ then $\{u_s\}$ defines a unitary action of $S$ on $A$.
\end{prop}

\begin{proof}
Suppose $(A,S,u)$ is a unitary action and $f\in p^*A$.  We need
to show that 
\[ h(s) := u_s f(s),\quad\text{and}\quad g(s) := u_s^*f(s) \]
define elements of $p^*A$.  The continuity of $h$ is obvious from
condition (c) of Definition~\ref{def:51}.  Suppose $s_i\rightarrow s$
and $a_i\rightarrow a$ such that $a_i\in A(p(s_i))$ for all $i$.  
First, observe that condition (b) of
Definition~\ref{def:51} guarantees that $u_{s\inv} = u_s\inv = u_s^*$
for all $s\in S$.  Therefore
\[
u_{s_i\inv}a_i = u_{s_i}^*a_i \rightarrow u_s^*a = u_{s\inv}a.
\]
It follows immediately that $g$ is continuous as well.  Furthermore, 
both $h$ and $g$ must vanish at infinity because $f$ does.
Thus $h,g\in p^*A$.  Hence \cite[Lemma 2]{yinglee} implies that there is
a multiplier $u$ such that $u(f)(s) = u_sf(s)$ for all $s\in S$.
Since each $u_s$ is a unitary, it is clear that $u$ must be a
unitary.  

Next, suppose we are given $u\in UM(p^*A)$.  Then, via 
\cite[Lemma 2]{yinglee}, we know there exists multipliers $u_s$ such that
$u_s(f(s)) = u(f)(s)$.  However, since $u$ is a unitary, each $u_s$
must be as well.  It is now straightforward, using \cite[Proposition
C.20]{tfb2}, to show that $\{u_s\}$ defines a unitary action of $S$ on $A$.
\end{proof}

Given a unitary dynamical system we can form an associated
groupoid dynamical system in the obvious way by using the adjoint
map. 

\begin{prop}
Suppose $(A,S,u)$ is a unitary dynamical system.  
Then the collection $\{\Ad u_s\}_{s\in S}$ defines a 
groupoid action of $S$ on $A$. We say that such an action is {\em
  unitary} or {\em  unitarily implemented}.
\end{prop}

\begin{proof}
Given a unitary action let $u$ be the corresponding element of
$UM(p^*A)$ guaranteed by Proposition \ref{prop:76}.  Then define
$\Ad u:p^*A\rightarrow p^*A$ by $\Ad u(f) = ufu^*$.  Clearly $\Ad u$ is a
$C_0(S\unit)$-linear automorphism of $p^*(A)$ and it is
straightforward to use \cite[Lemma 4.3]{renaultequiv} to show that
$\Ad u$ yields the desired dynamical system. 
\end{proof}

At this point we need to make a brief detour through the notion of
equivalent actions.  The following construction will
play the role of isomorphism. 

\begin{definition}
\label{def:53}
Suppose $G$ is a locally compact Hausdorff 
groupoid and $A$ is a $C_0(G\unit)$-algebra.  Furthermore,
suppose $\alpha$ and $\beta$ are actions of $G$ on $A$.  Then we say
that $\alpha$ and $\beta$ are {\em exterior equivalent} if there is a
collection $\{u_\gamma\}_{\gamma\in G}$ such that 
\begin{enumerate}
\item $u_\gamma \in UM(A(r(\gamma)))$ for all $\gamma\in G$, 
\item $u_{\gamma\eta} = u_\gamma\overline{\alpha}_\gamma(u_\eta)$ for
  all $\gamma,\eta\in G$ such that $s(\gamma)=r(\eta)$,
\item the map $(\gamma,a)\mapsto u_\gamma a$ is jointly continuous on
  the set $\{(\gamma,a)\in G\times\mcal{A}:{r(\gamma) = q(a)}\}$, and
\item $\beta_\gamma = \Ad u_\gamma \circ \alpha_\gamma$ for all
  $\gamma\in G$. 
\end{enumerate}
\end{definition}

As before, we present an alternate definition which removes the bundle theory. 

\begin{prop}
\label{prop:77}
Suppose $\alpha$ and $\beta$ are exterior equivalent actions of the
locally compact groupoid $G$ on the $C_0(G\unit)$-algebra $A$ with
the collection $\{u_\gamma\}$ implementing the equivalence.  Then
there is an element $u\in UM(r^*A)$ such that $u(f)(\gamma) = u_\gamma
f(\gamma)$ for all $f\in r^*A$ and $\gamma\in G$.  

Conversely, if $u\in UM(r^*A)$ then there are $u_\gamma\in
UM(A(r(\gamma)))$ for all $\gamma\in G$.  If $u_{\gamma\eta} =
u_\gamma \overline{\alpha}_\gamma(u_\eta)$ whenever $s(\gamma) =
r(\eta)$ and $\beta_\gamma = \Ad u_\gamma\circ\alpha_\gamma$ for all
$\gamma\in G$ then $\alpha$ and $\beta$ are exterior equivalent. 
\end{prop}

\begin{proof}
This is demonstrated in almost exactly the same way as Proposition
\ref{prop:76} and the proof is omitted for brevity. 
\end{proof}

The most important fact about exterior equivalent
actions is the following 

\begin{prop}
\label{prop:84}
Suppose $(A,G,\alpha)$ and $(A,G,\beta)$ are exterior equivalent
groupoid dynamical systems with the equivalence implemented
by $\{u_\gamma\}$.  Then the map
$\phi:\Gamma_c(G,r^*\mcal{A})\rightarrow \Gamma_c(G,r^*\mcal{A})$
defined by 
\begin{equation}
\label{eq:72}
\phi(f)(\gamma) = f(\gamma)u_\gamma^*
\end{equation}
for all $\gamma\in G$ extends to an isomorphism from $A\rtimes_\alpha
G$ onto $A\rtimes_\beta G$.  
\end{prop}

\begin{proof}
Let $(A,G,\alpha)$ and $(A,G,\beta)$ be exterior equivalent dynamical
systems with the equivalence implemented by $\{u_\gamma\}$.  Use
Proposition \ref{prop:77} to find $u\in UM(r^*A)$ such that
$uf(\gamma) = u_\gamma f(\gamma)$ for all $\gamma$.  Given $f\in
\Gamma_c(G,r^*\mcal{A})$ view $f$ as an element of the pull back 
$r^*A$ and define $\phi(f) = fu^*$.  
It is clear that 
$\phi:\Gamma_c(G,r^*\mcal{A})\rightarrow \Gamma_c(G,r^*\mcal{A})$
is given by \eqref{eq:72}.
Some lengthy, but unenlightening, computations show that $\phi$ is a
$*$-homomorphism with respect to the actions arising from $\alpha$ in
its domain and from $\beta$ on its range.  Since 
$\|\phi(f)(\gamma)\| = \|f(\gamma)\|$ for all $\gamma\in G$, 
it follows quickly that $\phi$ is continuous with respect to the
inductive limit topology.  Therefore the Disintegration Theorem 
\cite[Theorem 7.8,7.12]{renaultequiv} implies that 
$\phi$ extends to a $*$-homomorphism from $A\rtimes_\alpha G$
into $A\rtimes_\beta G$.  We can define an inverse $\psi$ for $\phi$ on
$\Gamma_c(G,r^*\mcal{A})$ by $\psi(f)(\gamma) = f(\gamma)u_\gamma$ in
an identical fashion.  
\end{proof}

Moving on, our statement that unitary actions are
``trivial'' dynamical systems will be supported by the next lemma.  However,
let us first introduce an action which is as simple as possible.  

\begin{example}
\label{ex:25}
Suppose $S$ is a locally compact group bundle and $A$ is a
$C_0(S\unit)$-algebra.  It is easy to show that the 
collection of identity maps
$\id_s:A(p(s))\rightarrow A(p(s))$ defines an action of $S$ on $A$. 
Observe that group bundles are the only groupoids which can act
trivially in this manner.  
\end{example}

\begin{lemma}
\label{lem:24}
If $(A,S,\alpha)$ is a unitary dynamical system then it is exterior
equivalent to the trivial system $(A,S,\id)$.  
\end{lemma}

\begin{proof}
Suppose $\alpha$ is implemented by the unitaries $\{u_s\}$.  Then an
elementary calculation shows that the 
$\{u_s\}$ also implement an equivalence between $\id$ and
$\alpha$.  
\end{proof}

\begin{remark}
The curious reader may wonder why we have only defined unitary actions
for a special class of groupoids.  Unitary actions should
always be equivalent to the trivial action and, as stated in Example
\ref{ex:25}, the trivial action only makes sense for group bundles.
Thus, it only makes sense to define unitary actions for group bundles.
\end{remark}

We want to show that crossed products of unitary dynamical systems are
tensor products.  However, we are working with fibred
objects so we need to use a tensor product which respects the bundle
structure on the algebras.  It is assumed that the reader is familiar
with the basics of balanced tensor products between $C_0(X)$-algebras,
although \cite{pullback} and \cite[Appendix B]{tfb} will serve as references.  

\begin{remark}
For completeness, let us recall that if $A$ and $B$ are
$C_0(X)$-algebras then the {\em balancing ideal} $I_X$ is the 
ideal in $A\otimes_{\max} B$ generated by 
\[
\{ f\cdot a \otimes b - a\otimes f\cdot b : f\in C_0(X), a\in A, b\in
B\}.
\]
The {\em balanced tensor product} $A\otimes_{C_0(X)} B$ is defined to
be the quotient $A\otimes_{\max} B/I_X$.  
Our tensor products will generally be maximal tensor products.    
However, most of the time we will be
working with nuclear $C^*$-algebras so that we will not have
to make this distinction.  
\end{remark}

We now prove the main theorem concerning unitary actions,
which is that they have trivial crossed products.  Because we are
working with universal completions, the proof is surprisingly
technical.  

\begin{theorem}
\label{thm:unitary}
Suppose $(A,S,\alpha)$ is a unitary dynamical system with
$\alpha$ implemented by $u$.  Then
there is a $C_0(S\unit)$-linear isomorphism $\phi:{C^*(S)\otimes_{C_0(S\unit)}
A}\rightarrow {A\rtimes_\alpha S}$ which is characterized for $a\in
A$ and $f\in C_c(S)$ by 
\begin{equation}
\phi(f\otimes a)(s) = f(s)a(p(s))u_s^*
\end{equation}
\end{theorem}

\begin{remark}
Since $\phi$ is $C_0(S\unit)$-linear it factors 
to isomorphisms $\phi_u:C^*(S_u)\otimes A(u) \rightarrow
A(u)\rtimes S_u$.  It is not difficult to check that these are the
usual isomorphisms that arise from unitary actions 
\cite[Lemma 2.73]{tfb2}.
\end{remark}

\begin{proof}
First, let $\beta$ be a Haar system for $S$ and 
consider the trivial action $\id$ of $S$ on $A$.  Given $a\in A$
and $f\in C_c(S)$ define $\iota(f\otimes a)(s) := f(s)a(p(s))$.
It follows easily 
that $\iota(f\otimes a)\in\Gamma_c(S,p^*\mcal{A})$.  Extend $\iota$
to the algebraic tensor product $C_c(S)\odot A$ by linearity so that
$\iota:C_c(S)\odot A\rightarrow \Gamma_c(S,r^*\mcal{A})$. Observe that
$\ran \iota$ is dense with respect to the inductive limit topology by 
\cite[Proposition 1.3]{inducpaper}.    Simple calculations on
elementary tensors show that $\iota$ is a $*$-homomorphism.  
Now we check that $\iota$ is bounded.  Suppose 
$(\pi,U,S\unit*\mfrk{H},\mu)$ is a covariant
representation of $(A,S,\id)$.  Then $U$ is a groupoid representation
of $S$ and we can form the integrated representation as in Remark \ref{rem:4},
which we also denote by $U$.  Let the collection $\{\pi_u\}_{u\in
  S\unit}$ be a decomposition of $\pi$ as in
Section \ref{sec:group-cross-prod}.  
Since $(\pi,U)$ is covariant, we must have, for all $a\in A$ and almost
every $s\in S$,
\begin{equation}
\label{eq:99}
\pi_{p(s)}(a(p(s)))U_s = U_s\pi_{p(s)}(a(p(s))).
\end{equation}
However, we can now compute for $f\in C_c(S)$ and 
$h\in \mcal{L}^2(S\unit*\mfrk{H},\mu)$ that 
\[
(\pi(a)U(f))h(u) 
= \int_S \pi_u(a(u))f(s)U(s)h(u)\Delta(s)\neghalf d\beta^u(s) 
= (U(f)\pi(a))h(u).
\]
We can extend this by continuity to all $f\in C^*(S)$ and conclude
that $\pi$ and $U$ are commuting representations of $A$ and $C^*(S)$.
It follows from \cite[Theorem B.27]{tfb} that there exists a
representation $U\otimes\pi$ on $C^*(S)\otimes A$ such that
$U\otimes \pi(f\otimes a) = U(f)\pi(a)$.
Given $f\in C_c(S)$ and $a\in A$ we check that 
\begin{align}
\label{eq:103}
\pi\rtimes U(\iota(f\otimes a))h(u) &= \int_S
\pi_u(f(s)a(u))U_sh(u)\Delta(s)\neghalf d\beta^u(s) \\ \nonumber
&= \pi_u(a(u))\int_S f(s)U_sh(u)\Delta(s)\neghalf d\beta^u(s) \\ \nonumber
&= \pi(a)U(f)h(u) = U\otimes \pi(f\otimes a)h(u).
\end{align}
Using linearity, we conclude that $\pi\rtimes U(\iota(\xi)) = U\otimes \pi(\xi)$
for all $\xi \in C_c(S)\odot A$.  Thus, given $\xi\in C_c(S)\odot A$, 
$\|\pi\rtimes U(\iota(\xi))\| = \|U\otimes \pi(\xi)\|
\leq \|\xi\|$.
Since this is true for all covariant representations, it
follows that $\iota$ is bounded and extends to a homomorphism on
$C^*(S)\otimes A$. Furthermore, since the range of $\iota$ is dense, it
must be surjective. What's more, given
$\phi\in C_0(S\unit)$, $f\in C_c(S)$ and $a\in A$ we have 
\[
\iota(\phi\cdot f\otimes a)(s) = \phi(p(s))f(s)a(p(s)) = 
\iota(f\otimes\phi\cdot a)(s).
\]
It follows by continuity and linearity that $\iota$ factors through
the balancing ideal and induces a surjective homomorphism
$\hat{\iota}:C^*(S)\otimes_{C_0(S\unit)} A\rightarrow A\rtimes S$.  

We would like to show that $\hat{\iota}$ is isometric.  Suppose $R$ is
a faithful representation of $C^*(S)\otimes_{C_0(S\unit)} A$ and let
$\overline{R}$ be its lift to $C^*(S)\otimes A$.  It follows
\cite[Corollary B.22]{tfb} that there are commuting representations $\pi$
and $U$ of $A$ and $C^*(S)$ such that $\overline{R}=U\otimes \pi$.
Furthermore, since $U\otimes \pi$ contains the balancing ideal, a
quick computation shows that $U(\phi\cdot f)\pi(a) =
U(f)\pi(\phi\cdot a)$ for all $\phi \in C_0(S\unit)$, $f\in C^*(S)$,
and $a\in A$.  Now, without loss of generality, we can use the
Disintegration Theorem to assume that $U$ is the integrated form of some
groupoid representation $(U,S\unit*\mfrk{H},\mu)$.  Furthermore we
have, for all $\phi\in C_0(S\unit)$, $a\in A$, and $f\in C_c(S)$
\[
\pi(\phi\cdot a)U(f)h(u) = \int_S \phi(u)f(s)U_s
\pi(a)h(u)\Delta(s)\neghalf d\beta^u(s)  = \phi(u)U(f)\pi(a)h(u).
\]
Since $U$ is nondegenerate, this implies that $\pi$ is
$C_0(S\unit)$-linear.  
Suppose the collection $\{\pi_u\}$ is a decomposition of $\pi$ and
let $\nu$ be the measure on $S$ induced by $\mu$.  All we need to do 
to prove that $(\pi,U)$ is a covariant representation of $(A,S,\id)$
is verify the covariance relation.  Let $\{a_i\}$ be a
countable dense subset in $A$ and $e_l$ a special orthogonal
fundamental sequence for $S\unit* \mfrk{H}$ \cite[Remark F.7]{tfb2}.  Since 
$\pi$ and $U$ commute, we have for all $i,l,k$ and $f\in C_c(S)$
\begin{align*}
0 =& (\pi(a_i)U(f)e_l,e_k) - (U(f)\pi(a_i)e_l,e_k) \\
=& \int_S (f(s)\pi_{p(s)}(a_i(p(s)))U_s
e_l(p(s)),e_k(p(s)))\Delta(s)\neghalf d\nu(s) \\
&- 
 \int_S
 (f(s)U_s\pi_{p(s)}(a_i(p(s)))e_l(p(s)),e_k(p(s)))\Delta(s)\neghalf
 d\nu(s) \\
=& \int_S f(s)((\pi_{p(s)}(a_i(p(s)))U_s -
U_s\pi_{p(s)}(a_i(p(s))))e_l(p(s)),e_k(p(s))) \Delta(s)\neghalf d\nu(s).
\end{align*}
This holds for all $f\in C_c(S)$ so that we may conclude for each $i,l$
and $k$ there exists a $\nu$-null set $N_{i,l,k}$ such that 
\begin{equation}
\label{eq:102}
((\pi_{p(s)}(a_i(p(s)))U_s -
U_s\pi_{p(s)}(a_i(p(s))))e_l(p(s)),e_k(p(s))) = 0
\end{equation}
for all $s\not\in N_{i,l,k}$.  However, if we let $N = \bigcup_{i,l,k}
N_{i,l,k}$ then $N$ is still a $\nu$-null set and for each $s\not\in
N$ \eqref{eq:102} holds for all $i,l$ and $k$.  Since $\{e_l(p(s))\}$
is a basis (plus zero vectors) for each $p(s)$, this implies that for
$s\not\in N$ we have 
$\pi_{p(s)}(a_i(p(s)))U_s = U_s \pi_{p(s)}(a_i(p(s)))$
for all $i$.  Because $\{a_i\}$ is dense in $A$,
this holds for all $a\in A$.  Thus $(\pi, U)$ is a covariant
representation of $(A,S,\id)$.  Furthermore, we can reuse the
computation in \eqref{eq:103} to conclude
that $\pi\rtimes U\circ \iota =
\pi\otimes U$.  Given $\xi\in C^*(S)\otimes A$ let $\xi'$ be its
image in $C^*(S)\otimes_{C_0(S\unit)} A$.  We then have
\[
\|\xi'\| = \|R(\xi')\| = \|U\otimes \pi(\xi)\| = 
\|\pi\rtimes U(\iota(\xi))\| \leq \|\iota(\xi)\| = \|\hat{\iota}(\xi')\|.  
\]
It follows that $\hat{\iota}$ is isometric and is therefore an
isomorphism.  

To finish the proof, observe that because of Proposition \ref{prop:84} and
Lemma \ref{lem:24}, the map $\psi:A\rtimes_{\id} S\rightarrow
A\rtimes_\alpha S$ given by $\psi(f)(s) = f(s)u_s^*$ is an
isomorphism.  Thus $\phi = \psi\circ \hat{\iota}$ is an isomorphism from
$C^*(S)\otimes_{C_0(S\unit)} A$ onto $A\rtimes_\alpha S$.  It follows
quickly that $\phi$ is $C_0(S\unit)$-linear and has the correct form. 
\end{proof}

\section{Locally Unitary Actions}
\label{sec:locally-unitary}

Now that we have developed the theory of unitary actions
we can modify Definition \ref{def:51} and introduce a new concept.
The basic idea is that we weaken the continuity condition and see what
kind of structure we have left.  

\begin{definition}
\label{def:56}
Suppose $S$ is a group bundle and $A$ is a $C_0(S\unit)$-algebra. A
dynamical system $(A,S,\alpha)$ is said to be {\em locally unitary} if
there is an open cover $\{U_i\}_{i\in I}$ of $S\unit$ such that
$(A(U_i),S|_{U_i},\alpha|_{S|_{U_i}})$ is unitarily implemented for all
$i\in I$.  
\end{definition}

Our goal will be to analyze the exterior equivalence classes of
abelian locally unitary
actions on $C^*$-algebras with Hausdorff spectrum.  In particular, the
rest of the $C^*$-algebras in this section will have Hausdorff
spectrum and we will view them as $C_0(\widehat{A})$-algebras in the
usual fashion.  

\begin{remark}
If $A$ has Hausdorff spectrum then $A$ is naturally a
$C_0(\widehat{A})$-algebra.  The fibres are given by $A(\pi) = A/\ker
\pi$ for all $\pi\in \widehat{A}$.  In particular, since we are
assuming separability, every fibre is elementary and isomorphic to the
compacts.  As such each fibre has a unique faithful irreducible representation
(up to equivalence).  
\end{remark}

With this assumption we obtain a very nice identification of the
spectrum for unitary crossed products.  

\begin{prop}
\label{prop:87}
Suppose $S$ is an abelian locally compact Hausdorff
group bundle with a Haar system, that $A$ is a $C^*$-algebra with
Hausdorff spectrum $S\unit$ and that $(A,S,\alpha)$ 
is a unitary dynamical system.  Let $\{u_s\}$
be the unitaries implementing $\alpha$ and for all $v\in S\unit$ let
$\pi_v$ be the unique irreducible representation
of $A(v)$. Define, for $\omega\in \widehat{S}$, 
\[
\omega\overline{\pi}_{\hat{p}(\omega)}(u)(s) :=
\omega(s)\overline{\pi}_{\hat{p}(\omega)}(u_s).
\]
Then the map $\phi :\widehat{S} \rightarrow
(A\rtimes_\alpha S)\sidehat$ given by $\phi(\omega) = \pi_{\hat{p}(\omega)}\rtimes
\omega \overline{\pi}_{\hat{p}(\omega)}(u)$ is a bundle homeomorphism.
\end{prop}

\begin{proof}
Let $(A,S,\alpha)$ and $u$ be as above. 
It follows from Theorem \ref{thm:unitary} that the map
$\psi:C^*(S)\otimes_{C_0(S\unit)} A\rightarrow A\rtimes S$ characterized
by $\psi(a\otimes f)(s) = f(s)a u_s^*$ is an isomorphism.  
Therefore, there is a homeomorphism $\phi_1:(C^*(S)\otimes_{C_0(S\unit)}
A)\sidehat\rightarrow (A\rtimes S)\sidehat$ such that
$\phi_1(R) = R\circ\psi\inv$.  Next, recall that we identify the dual
group bundle $\widehat{S}$ with $C^*(S)\sidehat$ \cite{bundleduality}.
Define $\widehat{S}\times_{S\unit} \widehat{A} := 
\{ (\omega,\pi_{\hat{p}(\omega)})\in\widehat{S}\times\widehat{A} :
\omega\in \widehat{S}\}$.
Since $C^*(S)$ is an abelian $C^*$-algebra, and is therefore GCR and
nuclear, it follows from \cite[Lemma 1.1]{pullback} that
$\phi_2:\widehat{S}\times_{S\unit}\widehat{A}\rightarrow(C^*(S)\otimes_{C_0(S\unit)}
A)\sidehat$ given by $\phi_2(\omega,\pi) = \omega\otimes_\sigma \pi$ 
is a homeomorphism.  Recall that if $\pi$ is a representation on $\mcal{H}$
then $\omega\otimes_\sigma \pi$ is the representation on $\mathbb{C} \otimes \mcal{H}$,
which we will of course identify with $\mcal{H}$, characterized by 
$\omega\otimes_\sigma \pi(f\otimes a) = \omega(f)\pi(a)$.  Moving on, since
$\widehat{A}= S\unit$ we can define another homeomorphism
$\phi_3:\widehat{S}\rightarrow \widehat{S}\times_{S\unit} \widehat{A}$
by $\phi_3(\omega)= (\omega,\pi_{\hat{p}(\omega)})$.  
Let $\phi = \phi_1\circ\phi_2\circ\phi_3$ and observe that
$\phi:\widehat{S}\rightarrow (A\rtimes S)\sidehat$ is a
homeomorphism.  Furthermore given $\omega\in \widehat{S}$ we have
$\phi(\omega) = \omega\otimes_\sigma \pi_{\hat{p}(\omega)}\circ\psi\inv$.  

Now fix $x\in S\unit$, $\omega \in \widehat{S}_x$ and define the map
$U:S_x\rightarrow U(\mcal{H})$ by $U_s = 
\omega(s)\overline{\pi}_x(u_s)$.  Since $u$ and $\omega$ are
continuous, it follows quickly that $U$ is a
unitary representation of $S_x$.  Furthermore, we can compute for $a\in
A(x)$ and $s\in S_x$ that 
\[
U_s\pi_x(a) = \omega(s)\pi_x(u_sa) = \omega(s)\pi_x(u_s a
u_s^*u_s) = \pi_x(\alpha_s(a))U_s.
\]
Thus $(\pi_x,U)$ is a covariant representation of
$(A(x),S_x,\alpha)$. As such we can form the integrated representation
$\pi_x\rtimes U$.   Recall that $A\rtimes S$ is a
$C_0(S\unit)$-algebra and that the restriction map $\rho$ 
factors to an isomorphism 
between $A\rtimes S(x)$ and $A(x)\rtimes S_x$.   Using the
restriction map to view  $\pi_x\rtimes U$ as a representation of
$A\rtimes S$ we claim that $\pi_x\rtimes U = \phi(\omega)$.
It will suffice to show that given an elementary tensor $f\otimes a$
then $\pi_x\rtimes U(\psi(f\otimes a)) = \omega\otimes_\sigma
\pi_x(f\otimes a)$. We compute, observing that the modular function is
one since $S$ is abelian,
\begin{align*}
\pi_x\rtimes U(\psi(f\otimes a))h
&= \int_S \pi_x(f(s)a(x)u_s^*)\omega(s) \overline{\pi}_x(u_s)h d\beta^x(s)
\\
&= \int_S f(s)\omega(s)d\beta^x(s) \pi_x(a(x))h \\
&= \omega(f)\pi_x(a)h = (\omega\otimes_\sigma \pi_x)(f\otimes a)h.
\end{align*}
Thus $\phi(\omega) = \pi_x\rtimes U$ and, since $U$ is just an abbreviated
notation for $\omega\overline{\pi}_x(u)$, we are done.
\end{proof}

\subsection{Characterization}

We saw in Proposition \ref{prop:87} that the spectrum of the unitary
crossed product was homeomorphic to $\widehat{S}$.  If the action were
locally unitary then it is interesting to ask if the spectrum is
locally homeomorphic to $\widehat{S}$ and if it is, in fact, a
principal $\widehat{S}$-bundle.  The answer is given in the following

\begin{theorem}
\label{thm:locunit}
Suppose $S$ is an abelian group bundle with a Haar system, 
that $A$ is a $C^*$-algebra with
Hausdorff spectrum $S\unit$ and that $(A,S,\alpha)$ 
is a locally unitary dynamical system.  Let $u^i$ implement $\alpha$
on $S|_{U_i}$ where $\{U_i\}$ is an open cover of $S\unit$ and let
$q:(A\rtimes_\alpha S)\sidehat\rightarrow S\unit$ be the bundle map.  
Then for each $i$ the map $\psi_i:\hat{p}\inv(U_i)\rightarrow
q\inv(U_i)$ such that 
\begin{equation}
\label{eq:105}
\psi_i(\omega) = \pi_{\hat{p}(\omega)}\rtimes
\omega\overline{\pi}_{\hat{p}(\omega)}(u^i)
\end{equation}
is a homeomorphism and the map $\gamma_{ij}$ such that 
\begin{equation}
\gamma_{ij}(p(s))(s) = \overline{\pi}_{p(s)}((u_s^i)^*u_s^j)
\end{equation}
defines a continuous section of $\widehat{S}$.  Furthermore, these maps
make $(A\rtimes S)\sidehat$ into a principal $\widehat{S}$-bundle with 
trivialization $(\mcal{U},\psi\inv,\gamma)$.  
\end{theorem}

\begin{proof}
Let $(A,S,\alpha)$ be as in the statement of the theorem.  Let
$\{u^i\}$ implement $\alpha$ on $S|_{U_i}$ where $U_i$ is an element of some
open cover $\mcal{U}$.  Given an open set $U\in\mcal{U}$ we identify
each of $(A(U)\rtimes S|_{U})\sidehat$, 
$(A\rtimes S(U))\sidehat$, and $q\inv(U)$ with the
disjoint union $\coprod_{x\in U}(A(x)\rtimes S_x)\sidehat$ as in
Remark \ref{rem:2}.
In a similar fashion we identify each of $(C^*(S)(U))\sidehat$,
$C^*(S|_U)\sidehat$ and $\hat{p}\inv(U)$ with the disjoint union
$\coprod_{x\in U}\widehat{S}_x$.

Now, fix $U_i\in\mcal{U}$.  By assumption 
$\alpha|_{S|_{U_i}}$, denoted $\alpha$ whenever possible, is
unitarily implemented by $\{u^i\}$ and as such Proposition
\ref{prop:87} implies that the map $\psi_i:(S|_{U_i})\sidehat
\rightarrow  (A(U_i)\rtimes S|_{U_i})\sidehat$ defined via
\eqref{eq:105} is a homeomorphism.  However, under the
identifications made in the previous paragraph, we can view $\psi_i$ as
a map from $\hat{p}\inv(U_i)$ onto $q\inv(U_i)$.  We
define the trivializing maps on $(A\rtimes S)\sidehat$ to be $\phi_i =
\psi_i\inv$.  What's more, since $(A\rtimes S)\sidehat$ is locally
homeomorphic to a locally compact Hausdorff space, we can conclude that
$(A\rtimes S)\sidehat$ is locally compact Hausdorff.  

Next, suppose $U_i,U_j\in\mcal{U}$ and for each $x\in
U_{ij}$ let $\pi_x$ be the (unique) irreducible representation of
$A(x)$.  On $A(x)\rtimes S_x$ both $u^i$ and $u^j$ implement $\alpha$
so that we compute, for $s\in S_x$ and $a\in A(x)$,
\[
\overline{\pi}_x((u_s^i)^*u_s^j)\pi_x(a) = 
\pi_x((u_s^i)^* u_s^j a) = \pi_x(\alpha_s\inv(\alpha_s(a))(u_s^i)^*u_s^j)
= \pi_x(a)\overline{\pi}_x((u_s^i)^*u_s^j).
\]
Since $\pi_x$ is irreducible, it follows \cite[Lemma A.1]{tfb} that
$\gamma_{ij}(x)(s):= \overline{\pi}_x((u_s^i)^*u_s^j)$ is a scalar.
Since $u_s^i$ and $u_s^j$ are unitaries, $\gamma_{ij}(x)(s)$ must be a
unitary as well and therefore has modulus one.  Some simple
computations show that $\gamma_{ij}(x)$ is a continuous homomorphism
so that $\gamma_{ij}(x)\in\widehat{S}_x$.  Thus $\gamma_{ij}$ is a
section of $\widehat{S}$ on $U_{ij}$ and we compute for
$\omega\in \hat{p}\inv(U_{ij})$
\begin{equation}
\label{eq:106}
\phi_i\circ\phi_j\inv(\omega) = \psi_i\inv\circ\psi_j(\omega) = 
\psi_i(\pi_{\hat{p}(\omega)}\rtimes
\omega\overline{\pi}_{\hat{p}(\omega)}(u^j)).
\end{equation}
Given $s\in S_{\hat{p}(\omega)}$ we have 
\begin{equation}
\label{eq:1}
\overline{\pi}_{\hat{p}(\omega)}(u_s^j)  = 
\overline{\pi}_{\hat{p}(\omega)}(u_s^i)\overline{\pi}_{\hat{p}(\omega)}((u_s^i)^*u_s^j)
= \gamma_{ij}(\hat{p}(\omega))(s)\overline{\pi}_{\hat{p}(\omega)}(u_s^i).
\end{equation}
Applying \eqref{eq:1} to \eqref{eq:106} we obtain
\[
\phi_i\circ\phi_j\inv(\omega) = 
\psi_i\inv(\pi_{\hat{p}(\omega)}\rtimes
(\omega\gamma_{ij}(\hat{p}(\omega))\overline{\pi}_{\hat{p}(\omega)}(u^i)))
= \omega\gamma_{ij}(\hat{p}(\omega))
\]
This shows that the $\gamma_{ij}$ are transition
functions for the $\phi_i$.  It follows quickly that $\gamma_{ij}$ is
continuous and that the trivialization $(\mcal{U},\phi,\gamma)$ makes
$(A\rtimes S)\sidehat$ into a principal $\widehat{S}$-bundle.
\end{proof}

Of course this is little more than a curiosity unless we can use the
principal bundle structure to tell us something about the action
$\alpha$.  Fortunately, we can do just that.  

\begin{theorem}
\label{thm:unique}
Suppose $S$ is an abelian group bundle with a Haar system 
and that $A$ has Hausdorff spectrum $S\unit$.
Two locally unitary actions $(A,S,\alpha)$ and $(A,S,\beta)$ are
exterior equivalent if and only if
$(A\rtimes_\alpha S)\sidehat$ and $(A\rtimes_\beta S)\sidehat$ are
isomorphic as $\widehat{S}$-bundles.  
\end{theorem}

\begin{remark}
Before we begin the proof of Theorem \ref{thm:unique}, let us give a
quick application to place it into context.  In \cite{specpaper} we
identified the spectrum of certain groupoid crossed products $A\rtimes
G$ as a quotient of the spectrum of the stabilizer crossed product 
$A\rtimes S$.  Coupled with
Theorem \ref{thm:unique}, this says that if the action of $S$ on $A$
is locally unitary then the spectrum of the global crossed product
$A\rtimes G$ is a quotient of a principal $\widehat{S}$-bundle and as
such has a cohomological invariant. 
\end{remark}

\begin{proof}
Suppose $\alpha$ and $\beta$ are equivalent locally unitary actions
and the equivalence is implemented by the collection $\{u_s\}$.  It
follows from Proposition \ref{prop:84} that the map
$\phi:A\rtimes_\alpha S\rightarrow A\rtimes_\beta S$ defined for $f\in
\Gamma_c(S,p^*\mcal{A})$ by $\phi(f)(s) = f(s)u_s^*$ is an
isomorphism.  As such it induces a homeomorphism $\Phi:(A\rtimes_\beta
S)\sidehat\rightarrow (A\rtimes_\alpha S)\sidehat$ via the map
$\Phi(\pi) = \pi\circ\phi$.  

Next, let us establish some notation.  Since $\alpha$ and $\beta$ are
both locally trivial we may as well pass to some common refinement and
assume that there exists an open cover $\mcal{U}$ of $S\unit$ such
that on $S|_{U_i}$ the unitary actions $v^i$ and $w^i$ implement
$\alpha$ and $\beta$, respectively.  Let $\phi_i$ and $\psi_i$ be the
trivializing maps induced by $v^i$ and $w^i$, respectively.
Furthermore, given $x\in S\unit$ let $\pi_x$ be the (unique) irreducible
representation of $A(x)$ associated to $x$.  Now fix
$U_i\in \mcal{U}$ and $x\in U_i$.  In order to conserve notation we will drop
the $i$'s on the $v^i$ and $w^i$ unless they are needed. 
Recall that $\beta_s = \Ad u_s\circ \alpha_s$ so that we can compute
for $s\in S_x$
\begin{align*}
u_s^* w_s v_s^* a &= u_s^*\beta_s(\alpha_s\inv(a))w_sv_s^* 
= \Ad(u_s^*)\circ \beta_s \circ\alpha_s\inv(a)u_s^*w_sv_s^* \\
&= \Ad(u_s^*)\circ\Ad(u_s) \circ \alpha_s \circ\alpha_s\inv(a) u_s^*w_s v_s^*
= a u_s^* w_s v_s^*.
\end{align*}
It follows that $\beta_i(x)(s) :=
\overline{\pi}_{x}(u_s^*w_sv_s^*)$ commutes with
$\pi_{x}(A(x))$. Since $\pi_{x}$ is irreducible, this implies
that $\beta_i(x)(s)$ must be a scalar.  As before, it is
straightforward to show that $\beta_i(x)$ is a continuous
$\mathbb{T}$-valued homomorphism and hence $\beta_i$ is a section of
$\widehat{S}$ on $U_i$.  
Given $\omega \in \widehat{S}_x$ we then compute for
$f\in \Gamma_c(S,p^*\mcal{A})$
\begin{align}
\label{eq:109}
\pi_x\rtimes (\omega\overline{\pi}_x(w))(\phi(f)) &= 
\int_S \pi_x(f(s))\omega(s)\overline{\pi}_x(u_s^*w_s)d\beta^x(x) \\ \nonumber
&= \pi_x\rtimes (\omega\beta_i(x) \overline{\pi}_x(v))(f).
\end{align}
We conclude from \eqref{eq:109} that 
\[
\phi_i\circ \Phi \circ \psi_i\inv(\omega) = 
\phi_i(\pi_x\rtimes (\omega\overline{\pi}_x(w))\circ \phi) 
= \phi_i(\pi_x \rtimes (\omega\beta_i(x) \overline{\pi}_x(v))) 
= \omega \beta_i(x).
\]
Therefore $\beta_i$ implements $\Phi$ on trivializations.
It is straightforward to show that $\beta_i$ is a continuous
section so that $(\mcal{U},\Phi,\beta)$ is an $\widehat{S}$-bundle
isomorphism.

Suppose that $(\mcal{U},\Phi,\beta)$ is an $\widehat{S}$-bundle
isomorphism of $(A\rtimes_\alpha S)\sidehat$ onto $(A\rtimes_\beta
S)\sidehat$.  Let $w^i$ and $v^i$ implement $\alpha$ and $\beta$,
respectively. Notice that
$\mcal{U}$ must be a common refinement of 
the local trivializing cover for $\alpha$ and $\beta$ so that we may
as well assume $w^i$ and $v^i$ are defined on $\mcal{U}$.
Fix $U_i\in \mcal{U}$ and $x\in U_i$.  
For each $s\in S_x$ we define a unitary $u_s\in
UM(A(x))$ by 
\begin{equation}
\label{eq:110}
u_s := \beta_i(x)(s)w_s^i(v_s^i)^*.
\end{equation}
We need to show that \eqref{eq:110} doesn't depend on the choice of
$U_i$. So suppose $x\in U_j$ as well.  Let $\gamma_{ij}$ and
$\eta_{ij}$ be the transition maps for $(A\rtimes_{\alpha} S)\sidehat$
and $(A\rtimes_\beta S)\sidehat$, respectively.  It follows from general
principal bundle nonsense that $\beta_i \gamma_{ij} =
\eta_{ij}\beta_i$.  We use this fact to compute
\begin{align}
\nonumber
\beta_i(x)(s)\overline{\pi}_x(w_s^i(v_s^i)^*) &= 
\beta_i(x)(s) \overline{\pi}_x(w_s^i(v_s^i)^*v_s^j(v_s^j)^*) 
= \beta_i(x)(s)\gamma_{ij}(x)(s)\overline{\pi}_x(w_s^i(v_s^j)^*) \\
\label{eq:111}
&= \beta_j(x)(s)\eta_{ij}(x)(s)\overline{\pi}_x(w_s^i(v_s^j)^*) \\\nonumber
&= \beta_j(x)(s)\overline{\pi}_x(w_s^i((w_s^i)^*w_s^j)(v_s^j)^*) 
= \beta_j(x)(s)\overline{\pi}_x(w_s^j(v_s^j)^*).
\end{align}
Since $\widehat{A}$ has Hausdorff spectrum, $A(x)$ is simple and
therefore \eqref{eq:111} implies
\[
\beta_i(x)(s)w_s^i(v_s^i)^* = \beta_j(x)(s)w_s^j(v_s^j)^*.
\]
Thus $u_s$ is well defined.  We now show that the $u_s$ implement
an equivalence between $\alpha$ and $\beta$.  The second condition
is the result of a simple computation. The continuity condition
is straightforward to prove using the fact that the actions $v$ and
$w$ are continuous, as well as the fact that $\beta_i$ is a continuous
section.  The last condition follows from the calculation
\[
\pi(\Ad u_s(\alpha_s(a))) = \beta_i(x)(s)\overline{\beta_i(x)(s)}
\pi(w_s v_s^* v_s a v_s^* v_s w_s^*) = \pi(w_s a w_s^*) = \pi(\beta_s(a)).
\]
Hence $\{u_s\}$ implements an equivalence between $\alpha$ and $\beta$. 
\end{proof}

Of course, this leads to the following

\begin{corr}
A locally unitary action of an abelian group bundle $S$ on a
$C^*$-algebra $A$ with Hausdorff spectrum $S\unit$ is determined, up to
exterior equivalence, by the associated cohomological invariant of
$(A\rtimes S)\sidehat$ as a principal $\widehat{S}$-bundle.
Furthermore, this cohomology class is an invariant for the isomorphism class of $A\rtimes S$.
\end{corr}

\subsection{Existence}

The final piece of the puzzle will be to prove that locally unitary
actions are about as abundant as they can be. In other words, we will
show that every
principal bundle can be obtained through a locally unitary action.  

\begin{theorem}
\label{thm:exist}
Suppose $S$ is an abelian group bundle with a Haar system 
and $q:X\rightarrow S\unit$ is a principal $S$-bundle. 
Then $C_0(X)\rtimes S$ has
Hausdorff spectrum $S\unit$ and the dual action of $\widehat{S}$ on
$C_0(X)\rtimes S$ defined for $\omega\in \widehat{S}_u$ 
and $f\in C_c(S_u\times X_u)$ by 
\begin{equation}
\label{eq:112}
\widehat{\lt}_\omega(f)(s,x) = \omega(s)f(s,x)
\end{equation}
is locally unitary.  Furthermore, $((C_0(X)\rtimes S)\rtimes
\widehat{S})\sidehat$ and $X$ are isomorphic $S$-bundles. 
\end{theorem}

We begin by proving the following

\begin{prop}
\label{prop:90}
Suppose $S$ is an abelian group bundle with a Haar system and that
$X$ is a principal $S$-bundle.  
Then $C_0(X)\rtimes S$ has Hausdorff spectrum $S\unit$. 
\end{prop}

\begin{proof}
We know from \cite[Proposition 1.2]{specpaper} 
that $C_0(X)\rtimes S$ is a
$C_0(S\unit)$-algebra with fibres $C_0(X_u)\rtimes S_u$.  Hence there 
is a continuous surjection $r$ of
$(C_0(X)\rtimes S)\sidehat$ onto $S\unit$.  Furthermore, we may 
identify $r\inv(u)$ with $(C_0(X_u)\rtimes S_u)\sidehat$ in the usual fashion.  
Next, let $\phi:X_u\rightarrow S_u$ be the restriction of one of the
trivializing maps to $X_u$.  Since $\phi$ is a homeomorphism,
we can pull back the group structure from $S_u$ to $X_u$ and turn
$\phi$ into a group isomorphism.  Furthermore, it follows from
Proposition \ref{prop:27} that $\phi$ is equivariant with respect to
the action of $S_u$ on $X_u$.  Therefore
if we identify $X_u$ with $S_u$ then the action of $S_u$ on
$X_u$ becomes the action of $S_u$ on itself by translation.  In other
words, $C_0(X_u)\rtimes S_u$ is isomorphic to $C_0(S_u)\rtimes_{\lt}
S_u$.  We know from the Stone-von Neumann Theorem \cite{histstonevon} 
that $C_0(S_u)\rtimes S_u$ is
isomorphic to the compact operators on some separable Hilbert space.
Hence $C_0(S_u)\rtimes S_u$, and therefore $C_0(X)\rtimes S(u)$, has a unique
irreducible representation.  It follows that the map $r$ is
injective.  

All that remains is to show that $r$ is open, or equivalently,
closed.  Suppose $C$ is a closed subset of $(C_0(X)\rtimes S)\sidehat$.  Then
there is some ideal $I$ such that $C = \{\pi\in (C_0(X)\rtimes S)\sidehat :
I\subset \ker \pi\}$. Let $D = \{ u\in S\unit : I\subset I_u\}$
where the ideal $I_u$ in $C_0(X)\rtimes S$ is given by 
\[
I_u = \cspn\{\phi\cdot f:
\phi\in C_0(S\unit), f\in C_0(X)\rtimes S, \phi(u) = 0\}.
\]  
It is straightforward to show that $D=r(C)$.  We will limit ourselves
to proving that $D$ is closed.  Suppose $u_i\rightarrow u$ in $S\unit$ and
$u_i\in D$ for all $i$.  Then, since $I\subset I_{u_i}$ for all $i$, we
have $f(u_i) = 0$ for all $i\in I$.  However, $f$ is continuous 
when viewed as a function on $S\unit$ so that $f(u) = 0$.  Thus $f\in
I_u$ and $u\in D$.  
\end{proof}

Next, we show that there is a dual action of $\widehat{S}$ on
$C_0(X)\rtimes S$ induced by left translation.  Since it isn't much harder, 
we actually prove this result in greater generality. Unfortunately,
we can't just jump right in.  Verifying the continuity condition will
take work.  In particular, we have to deal with the topology on the
usc-bundle associated to $A\rtimes S$.  

\begin{lemma}
\label{lem:39}
Suppose $(A,S,\alpha)$ is a  dynamical system and that 
$S$ is an abelian group bundle.  Let $\mcal{A}$ be the
usc-bundle associated to $A$, define
\[
\widehat{S}*S = \{(\omega,s)\in \widehat{S}\times S : \hat{p}(\omega)
= p(s)\},
\]
and let $p:\widehat{S}*S \rightarrow S\unit$ be given by $p(\omega,s)
= p(s)$.  Then there is a map
$\iota:\Gamma_c(\widehat{S}*S,p^*\mcal{A})\rightarrow \hat{p}^*
(A\rtimes S)$ such that $\iota(f)(\omega)(s) = f(\omega, s)$.
Furthermore, $\iota$ is continuous with respect to the inductive limit
topology and the range of $\iota$ is dense.  
\end{lemma}

\begin{proof}
The only difficult part is showing that $\iota(f)$ is
continuous as a function into $\mcal{E}$ where $\mcal{E}$ is the
usc-bundle associated to $A\rtimes S$.  
We start out with a simpler function.  Suppose $g\in
C_c(\widehat{S})$, $h\in C_c(S)$ and $a\in A$.  Define $g\otimes
h\otimes a$ on $\widehat{S}*S$  by $g\otimes h \otimes a(\omega,s) =
g(\omega)h(s)a(p(s))$.  It is clear that $g\otimes h\otimes a\in
\Gamma_c(\widehat{S}*S,p^*\mcal{A})$.  Furthermore, if we view
$h\otimes a$ as an element of $\Gamma_c(S,p^*\mcal{A})$ then
$\iota(g\otimes h\otimes a)(\omega) = g(\omega) (h\otimes a)(\hat{p}(\omega))$
where $(h\otimes a)(\hat{p}(\omega))$ is just the restriction of
$h\otimes a$ to $S_{\hat{p}(\omega)}$.  Since $h\otimes a$ defines a
continuous section of $\mcal{E}$, it is easy to see that
$\iota(g\otimes h\otimes a)$ is a continuous function from
$\widehat{S}$ into $\mcal{E}$.  Thus $\iota(g\otimes h\otimes a) \in
\Gamma_c(\widehat{S},\hat{p}^*\mcal{E})$.  

We now show $\iota$ preserves convergence with respect to the
inductive limit topology.  Suppose $f_i\rightarrow f$ uniformly in
$\Gamma_c(\widehat{S}*S,p^*\mcal{A})$ and that eventually the supports
are contained in some fixed compact set $K$.  Clearly the supports of
$\iota(f)$ are eventually contained in the projection of $K$ to
$\widehat{S}$.  Fix $\epsilon > 0$ and let $M$ be an upper bound for
$\{\beta^u(L)\}$ where $L$ is the projection of $K$ to $S$.  Then
eventually $\|f_i-f\|_\infty < \epsilon/M$.  Therefore for large $i$ we
have, given $\omega \in \widehat{S}_u$ and making use of the fact that
$S_u$ is abelian so the $I$-norm on $C_c(S_u,A(u))$ only has one term, 
\begin{align*}
\|\iota(f_i)(\omega)-\iota(f)(\omega)\|&\leq \|\iota(f_i)(\omega) -
\iota(f)(\omega)\|_I \\
&= \int_S \|f_i(\omega,s) - f(\omega,s)\| d\beta^u(s) \\
&\leq \|f_i-f\|_\infty M < \epsilon 
\end{align*}
Thus $\iota(f_i)\rightarrow \iota(f)$ uniformly and hence with
respect to the inductive limit topology.  

Now suppose $f\in \Gamma_c(\widehat{S}*S,p^*\mcal{A})$ and that
$\omega_i\rightarrow \omega$ in $\widehat{S}$. Fix $\epsilon > 0$ and
let $U$ and $V$ be relatively compact neighborhoods of the projection
of $\supp f$ to $\widehat{S}$ and $S$, respectively.  
Since sums of elementary tensors are dense \cite[Proposition
1.3]{inducpaper} we may 
find $\{g_i\}_{i=1}^N\in C_c(\widehat{S}*S)$ and $\{a_i\}_{i=1}^N\in A$
such that $\|f-\sum_i g_i\otimes
a_i\|_\infty < \epsilon/2$.  For each $1\leq i \leq N$ extend $g_i$ to all of
$C_c(\widehat{S}\times S)$ and choose $h_i^j\in C_c(\widehat{S})$ and
$k_i^j\in C_c(S)$ such that 
$\|g_i-\sum_jh_i^j\otimes k_i^j\|_\infty < \epsilon/(2N\|a_i\|)$.
It then follows from some simple computations that 
\[
\left\|f - \sum_{i=1}^N\sum_j h_i^j\otimes k_i^j \otimes
  a_i\right\|_\infty \leq  \epsilon/2 +
\sum_{i=1}^N\|a_i\|\left\|g_i-\sum_j h_i^j\otimes k_i^j\right\|_\infty
< \epsilon. 
\]
Furthermore, we can multiply the $h_i^j$ and $k_i^j$ by 
functions which vanish off
$U$ and $V$, respectively, so that $\supp h_i^j\otimes k_i^j\otimes a \subset
\overline{U}\times \overline{V}$.  This construction shows that
sums of elements of the form $h\otimes k\otimes a$ for $h\in
C_c(\widehat{S})$, $k\in C_c(S)$, and $a\in A$ are dense in
$\Gamma_c(\widehat{S}*S, p^*\mcal{A})$ with respect to the inductive
limit topology.  

At last we can show that $\iota(f)$ is continuous for 
$f\in \Gamma_c(\widehat{S}*S,p^*\mcal{A})$.  Let $g_i = \sum_k h_i^k\otimes
k_i^k\otimes a_i^k$ 
be a sequence converging to $f$ in the inductive limit topology as
above.  Since sums of continuous functions are continuous, 
$\iota(g_i)(\omega_j)\rightarrow \iota(g_i)(\omega)$ for all $i$ and it now
follows from a straightforward application of \cite[Proposition
C.20]{tfb2} that $\iota(f)(\omega_i)\rightarrow \iota(f)(\omega)$.
Thus $\iota(f)$ is a continuous section.  Showing that $\iota$ has
dense range now follows from \cite[Proposition C.24]{tfb2} after a
brief argument.  
\end{proof}

The following corollary isn't necessary to build the dual action, but
it will be needed in the proof of Theorem \ref{thm:exist} so we
include it here.  

\begin{corr}
\label{cor:12}
Suppose $S$ is an abelian group bundle with a Haar system 
and $q:X\rightarrow S\unit$ is a principal $S$-bundle. Define 
\[
\widehat{S}*S*X := \{(\omega,s,x)\in \widehat{S}\times S\times X :
\hat{p}(\omega) = p(s) = q(x)\}
\]
Then there is a map $\iota:C_c(\widehat{S}*S*X)\rightarrow
\hat{p}^*(C_0(X)\rtimes S)$ such that 
\[
\iota(f)(\omega)(s)(x) = f(\omega,s,x).
\]
Furthermore, $\iota$ is continuous with respect to the inductive limit
topology and the range of $\iota$ is dense. 
\end{corr}

\begin{proof}
Let $\mcal{C}$ be the usc-bundle associated to $C_0(X)$ as a
$C_0(S\unit)$-algebra.  Consider the map
$\iota_1:\Gamma_c(\widehat{S}*S,p^*\mcal{C})\rightarrow
\hat{p}^*(C_0(X)\rtimes S)$ given by 
$\iota_1(f)(\omega)(s)(x) := f(\omega,s)(x)$.  
It follows from Lemma \ref{lem:39} that this map is
continuous with respect to the inductive limit topology and its range
is dense in $\hat{p}^*(C_0(X)\rtimes S)$.  Now consider the map
$\iota_2:C_c(\widehat{S}*S*X)\rightarrow
\Gamma_c(\widehat{S}*S,p^*\mcal{C})$ given by $\iota_2(f)(\omega,s)(x)
= f(\omega,s,x)$. It follows from Lemma \ref{lem:26} that $\iota_2$ is
surjective and preserves the inductive limit topology.  Thus
the map $\iota = \iota_2\circ\iota_1$ has the correct form and all the
right properties. 
\end{proof}

Now we can finally tackle the dual action construction.  This will
provide the last tool we need to demonstrate Theorem \ref{thm:exist}.  

\begin{prop}
\label{prop:89}
Suppose $(A,S,\alpha)$ is a dynamical system and that $S$ is
an abelian group bundle with a Haar system.  
Then for each $\omega\in \widehat{S}$ there is
an automorphism $\hat{\alpha}_\omega$ on $A\rtimes S(\hat{p}(\omega))$
defined for $f\in C_c(S_{\hat{p}(\omega)},A(\hat{p}(\omega)))$ by 
\[
\hat{\alpha}_\omega(f)(s) = \overline{\omega(s)}f(s).
\]
With this action $(A\rtimes S, \widehat{S},\hat{\alpha})$ is a
dynamical system. 
\end{prop}

\begin{proof}
Since everything else is straightforward, we will limit ourselves to
demonstrating the continuity of the action.  
Let $\mcal{E}$ be the bundle associated to $A\rtimes S$ and suppose 
$\omega_i\rightarrow \omega$ in $\widehat{S}$ and $f_i\rightarrow f$
in $\mcal{E}$ such that $f_i\in A\rtimes S(\hat{p}(\omega_i))$ for all
$i$.  Now choose $g\in \hat{p}^*(A\rtimes S)$ such that
$g(\omega) = f$.  It follows from Lemma \ref{lem:39} that we can
choose $h\in \Gamma_c(\widehat{S}*S,p^*\mcal{A})$ such that
$\|\iota(h)-g\|_\infty < \epsilon/2$.  Define $\alpha(h)(\omega,s) =
\overline{\omega(s)} h(\omega,s)$.  It is clear that $\alpha(h)\in
\Gamma_c(\widehat{S}*S,p^*\mcal{A})$.  It is also easy to see that
$\iota(\alpha(h))(\omega) = \alpha_\omega(\iota(h)(\omega))$.  Thus
$
\|\iota(\alpha(h))(\omega)-\alpha_\omega(f)\| =
\|\alpha_\omega(\iota(h)(\omega)-g(\omega))\| < \epsilon/2 < \epsilon$.
Next, since $g(\omega_i)\rightarrow g(\omega) =f$ and $f_i\rightarrow
f$ we have $\|g(\omega_i)-f_i\|\rightarrow 0$.  Therefore, eventually,
we have 
\begin{align*}
\|\iota(\alpha(h))(\omega_i)-\alpha_{\omega_i}(f_i)\| &\leq
\|\alpha_{\omega_i}(\iota(h)(\omega_i)-g(\omega_i))\| +
\|\alpha_{\omega_i}(g(\omega_i)-f_i)\| \\
&\leq \epsilon/2 + \|g(\omega_i)-f_i\| < \epsilon.
\end{align*}
It follows from \cite[Proposition C.20]{tfb2} that
$\alpha_{\omega_i}(f_i)\rightarrow \alpha_\omega(f)$.  
\end{proof}

\begin{remark}
The action from Proposition \ref{prop:89} is a generalization of the
usual Takai dual action for abelian groups.  
\end{remark}

We are now ready to prove our existence theorem. 

\begin{proof}[Proof of Theorem \ref{thm:exist}]
We have shown in Proposition \ref{prop:90} 
that $C_0(X)\rtimes S$ has Hausdorff spectrum $S\unit$.
Furthermore, we showed in Proposition \ref{prop:89} that
there is an action of $\widehat{S}$ which, if we view $C_c(S_u\times
X_u)$ as sitting densely inside $C_c(S_u,C_0(X_u))$, is given by 
\[
\widehat{\lt}_\omega(f)(s,x) = \overline{\omega(s)}f(s,x)
\]
for $f\in C_c(S_u\times X_u)$.  We need to show that $\widehat{\lt}$
is locally unitary. Let $\mcal{U}$ be a trivializing cover of $X$ and
let $\phi_i$ be the local trivializations.  
Fix $U_i\in\mcal{U}$.  Then for all 
$w\in U_i$, $\omega \in \widehat{S}_w$, and $f\in C_c(S_w\times X_w)$
define 
\begin{equation}
\label{eq:108}
u_\omega f(s,x) := \overline{\omega(\phi_i(x))}f(s,x).
\end{equation}
Simple calculations show that 
$u$ is a homomorphism on $S_w$.  Next we will show that $u$ is
adjointable.  Equip $C_0(X_w)\rtimes S_w$ with is usual inner product
as a Hilbert module.  For all $f,g\in C_c(S_w\times X_w)$ we have 
\begin{align*}
\langle u_\omega f,g\rangle(s,x) 
&= \int_S \omega(\phi_i(t\inv\cdot x)) \overline{f(t\inv,t\inv\cdot x)}g(t\inv
s, t\inv\cdot x)d\beta^w(t) \\ &= \langle f,u_{\omega\inv}g\rangle(s,x).
\end{align*}
This shows that $u_\omega$ is adjointable on $C_c(S_w\times X_w)$ and we can
also observe that 
\[
\|u_\omega f\|^2 = \|\langle u_\omega f, u_\omega f\rangle\| = 
\|\langle f, u_{\omega\inv}u_\omega f\rangle \| = \|\langle f, f\rangle\| = \|f\|^2.
\]
Thus $u_\omega$ is isometric on $C_c(S_w\times X_w)$ and as such it can be
extended to a unitary multiplier on $C_0(X_w)\rtimes S_w$.  
All that remains for the collection
$\{u_\omega\}$ to define a unitary action of $\hat{p}\inv(U_i)$ on
$C_0(X)\rtimes S(U_i)$ is continuity.  

Let $\mcal{E}$ be the bundle associated to $C_0(X)\rtimes S$ and fix
$\epsilon > 0$.  Suppose $\omega_j\rightarrow \omega$ in $\hat{p}\inv(U_i)$ and
$f_j\rightarrow f$ in $\mcal{E}|_{U_i}$ such that $f_j\in
C_0(X)\rtimes S(\hat{p}(\omega_j))$ for all $j$.  Choose $g\in \hat{p}^*
(C_0(X)\rtimes S)$ such that $g(\omega)=f$.   Using Corollary \ref{cor:12} we can find a
continuous, compactly supported function $h$ on $\widehat{S}*S*X$ such that
$\|\iota(h)-g\|_\infty < \epsilon/2$.  Consider the open set $O =
\hat{p}\inv(U_i)*p\inv(U_i)*q\inv(U_i)$ in $\widehat{S}*S*X$. We
define a new function $k\in C(O)$ by 
\[
k(\chi,s,x) = \psi(p(s))\overline{\chi(\phi_i(x))}h(\chi,s,x)
\]
where $\psi\in C_c(U_i)$ is some function which is one on a
neighborhood of $\hat{p}(\omega)$. Now $k$ is clearly compactly
supported with $\supp k \subset O$.
Therefore we can, and do, extend $k$ by zero to all of $\widehat{S}*S*X$.  Next
we observe the following facts.  First, 
$\iota(k)(\omega) = u_\omega \iota(h)(\omega)$,
and eventually $\iota(k)(\omega_j) =
u_{\omega_j} \iota(h)(\omega_j$).  Second, that 
\[
\|\iota(k)(\omega) - u_\omega f\| = \|u_\omega(\iota(h)(\omega) -
g(\omega))\| = \|\iota(h)(\omega) - g(\omega)\| < \epsilon/2.
\]
Third, $f_i\rightarrow f$ and $g(\omega_i)\rightarrow g(\omega) =
f$ so that $\|f_i-g(\omega_i)\|\rightarrow 0$.  Thus, eventually,
we have 
\begin{align*}
\|\iota(k)(\omega_i) - u_{\omega_i}f_i\| &\leq
\|u_{\omega_i}(\iota(h)(\omega_i) - g(\omega_i)) \| +
\|u_{\omega_i}(g(\omega_i) - f_i)\| \\
&\leq \epsilon/2 + \|g(\omega_i)-f_i\| < \epsilon.
\end{align*}
Finally, we observe that $\iota(k)(\omega_i)\rightarrow
\iota(k)(\omega)$ since $\iota(k)$ is a continuous section.  It
follows from \cite[Proposition C.20]{tfb2} 
that $u_{\omega_i}f_i\rightarrow u_\omega f$ and hence
$\{u_\omega\}$ defines a unitary action of $\hat{p}\inv(U_i)$ on
$C_0(X)\rtimes S(U_i)$.  What's more, the calculation 
\begin{align*}
u_\omega f u_\omega^*(s,x) &= \overline{\omega(\phi_i(x))} (u_\omega
f^*)^*(s,x) = \overline{\omega(\phi_i(x)) u_\omega
  f^*(s\inv,s\inv\cdot x)} \\
&= \overline{\omega(\phi_i(x))} \omega(\phi_i(s\inv\cdot x)) f(s,x) = 
\overline{\omega(\phi_i(x))}\overline{\omega(s)}\omega(\phi_i(x))
f(s,x) \\
&= \overline{\omega(s)}f(s,x) = \widehat{\lt}f(s,x)
\end{align*}
shows that $u$ implements $\widehat{\lt}$ on $\hat{p}\inv(U_i)$.  
Since we performed this construction for an arbitrary element 
of the cover $\mcal{U}$, it follows that $\widehat{\lt}$ is locally
unitary.  

Consider $Y = ((C_0(X)\rtimes S)\rtimes \widehat{S})\sidehat$.  Now, $Y$ is
a principal $\doubledual{S}$-bundle and, in light of duality
\cite[Theorem 16]{bundleduality}, a principal $S$-bundle as well.  We would
like to show that $Y$ is isomorphic to $X$.  
Using Theorem \ref{prop:principcohom} it suffices to
show that $X$ and $Y$ have the same cohomological invariant.  Let
$\gamma_{ij}$ be the transition functions for $X$ with respect to the
trivializing maps $\phi_i$.  Let $\eta_{ij}$ be the transition
functions for $Y$ and recall from Theorem \ref{thm:locunit} that we have 
$\eta_{ij}(v)(\omega) = \overline{\pi}_v((u_\omega^i)^*u_\omega^j)$
where $\pi_v$ is the unique irreducible representation of
$A(v)$ and $u_\omega^i, u_\omega^j$ are the unitaries
constructed above.  It follows from general principal bundle nonsense
that $\phi_i(x) = \gamma_{ij}(q(x))\phi_j(x)$ for all $x$.
We now compute for $f\in C_c(S_v\times X_v)$ 
\begin{align*}
((u_\omega^i)^*u_\omega^j f)(s,x) &=
\overline{\omega\inv(\phi_i(x))\omega(\phi_j(x))}f(s,x) =
\omega(\gamma_{ij}(v)\phi_j(x))\overline{\omega(\phi_j(x))}f(s,x)\\
&= \omega(\gamma_{ij}(v))f(s,x) = \hat{\gamma}_{ij}(v)(\omega) f(s,x)
\end{align*}
where $\hat{\gamma}_{ij}(v)$ denotes the image of $\gamma_{ij}(v)$ in
the double dual.  Therefore, since $\pi_v$ is faithful,
$\eta_{ij}(v)(\omega) = \hat{\gamma}_{ij}(v)(\omega)$.  Thus, once we
identify $S$ with $\doubledual{S}$, the cohomological invariants of
$X$ and $Y$ are identical.  
\end{proof}

\begin{example}
Theorem \ref{thm:exist} says that any
principal $S$-bundle $X$ gives rise to a locally unitary action of
$\widehat{S}$ on $C_0(X)\rtimes S$ and in particular this holds for locally
$\sigma$-trivial bundles.  Thus any example in \cite{locunitarystab}
yields a locally unitary action.  
\end{example}

\begin{remark}
It is worth describing, at least briefly, how this material
generalizes \cite{locunitary}.  
Suppose $H$ is an abelian group and $A$ has
Hausdorff spectrum $X$.  If $\alpha$ is an action of $H$ on $A$
then, as in \cite[Example 1.4]{inducpaper}, we can form the transformation
groupoid $H\times X$ and there is an action $\beta$ of $H\times X$ on
$A$.  Furthermore, we have $A\rtimes_\alpha H \cong A\rtimes_\beta
(H\times X)$.  Without getting into the details, $\alpha$ is locally
unitary according \cite{locunitary} if, for each $\pi\in X$, there is an
open neighborhood $U$ of $\pi$ and a strictly continuous map $u:H\rightarrow
M(A)$ such that for each $\rho\in U$, $\bar{\rho}\circ u$ is a
representation of $H$ on $\mcal{H}_\rho$ which implements $\alpha$.  In
particular this implies that $\rho\circ \alpha_s\inv =
\bar{\rho}(u_s)\rho \bar{\rho}(u_s^*)$ is equivalent to $\rho$.
Thus the action of $H$ on $X$ induced by $\alpha$ is trivial and the
transformation groupoid $H\times X$ is the trivial
group bundle.  What's more, since $A$ has Hausdorff spectrum, 
it is not hard to show
that $u_s(x)$ implements $\beta_{(s,x)}$ on $A(x)$ and that $\beta$ is
unitarily implemented by $\{u_s(x)\}$ on $H\times U$.  Thus $\beta$ is
a locally unitary action of $H\times S$ on $A$.  Now,
the dual of $H\times X$ is $\widehat{H}\times X$
and we have, according to Theorem \ref{thm:locunit}, that
$(A\rtimes_\alpha H)\sidehat \cong
(A\rtimes_\beta(H\times X))\sidehat$ is a principal $\widehat{H}\times
X$ bundle.  However, it follows from Remark \ref{rem:1} that this
implies $(A\rtimes_\alpha H)\sidehat$ is a principal
$\widehat{H}$-bundle.  From here it is straightforward to see how the results of
this section are related to those in \cite{locunitary}.
\end{remark}

\bibliographystyle{amsplain}
\bibliography{/home/goehle/references}

\end{document}